\documentclass[a4paper]{amsart}

\usepackage[utf8]{inputenc}
\usepackage[T1]{fontenc}
\usepackage[american]{babel}
\usepackage{amsmath,amssymb,amsthm}
\usepackage{enumerate}
\usepackage[left=3cm,right=3cm,top=3cm,bottom=3cm]{geometry}
\usepackage{dsfont}
\usepackage{csquotes}
\usepackage[style=alphabetic,sorting=anyt,backend=biber,doi=false,url=false,maxbibnames=10,isbn=false]{biblatex}
\usepackage[hidelinks]{hyperref}
\usepackage{cleveref}
\usepackage{comment}

\newcommand{\1}{\mathds{1}}

\newcommand{\F}{\mathcal F}
\renewcommand{\H}{\mathcal H}
\newcommand{\J}{\mathcal J}
\renewcommand{\L}{\mathcal L}
\renewcommand{\S}{\mathcal S}
\newcommand{\T}{\mathcal T}
\newcommand{\U}{\mathcal U}
\newcommand{\V}{\mathcal V}

\newcommand{\IC}{\mathbb C}
\newcommand{\IR}{\mathbb R}

\newcommand{\m}{\mathfrak m}
\newcommand{\n}{\mathfrak n}

\newcommand{\dom}{\operatorname{dom}}

\newcommand{\id}{\mathrm{id}}

\newcommand{\op}{\mathrm{op}}

\newcommand{\abs}[1]{\lvert#1\rvert}
\newcommand{\norm}[1]{\lVert#1\rVert}

\renewcommand{\epsilon}{\varepsilon}
\renewcommand{\phi}{\varphi}

\newtheorem{proposition}{Proposition}[section]
\newtheorem{lemma}[proposition]{Lemma}
\newtheorem{theorem}[proposition]{Theorem}
\newtheorem{corollary}[proposition]{Corollary}

\theoremstyle{definition}
\newtheorem{definition}[proposition]{Definition}

\theoremstyle{remark}
\newtheorem{example}[proposition]{Example}
\newtheorem{remark}[proposition]{Remark}

\title{Operator-Valued Twisted Araki--Woods Algebras}
\author{Rahul Kumar R}
\address{(R.K.) Department of Mathematics and Statistics, IIT Kanpur, Kalyanpur - 208016, Uttar Pradesh, India}
\email{rahulkumarr35@gmail.com}

\author{Melchior Wirth}
\address{(M. W.) Institute of Science and Technology Austria (ISTA), Am Campus 1, 3400 Klosterneuburg, Austria}
\email{melchior.wirth@ist.ac.at}

\begin{document}
\begin{abstract}
We introduce operator-valued twisted Araki--Woods algebras. These are operator-valued versions of a class of second quantization algebras that includes $q$-Gaussian and $q$-Araki--Woods algebras and also generalize Shlyakhtenko's von Neumann algebras generated by operator-valued semicircular variables. We develop a disintegration theory that reduces the isomorphism type of operator-valued twisted Araki--Woods algebras over type I factors to the scalar-valued case. Moreover, these algebras come with a natural weight, and we characterize its modular theory. We also give sufficient criteria that guarantee that factoriality of these algebras.
\end{abstract}
\maketitle

\section{Introduction}
Second quantization von Neumann algebras or Gaussian type von Neumann algebras are algebras generated by field operators on various versions of Fock spaces. They arise naturally both from a purely mathematical and a physics viewpoint and play an important role in the study of operator algebras. Perhaps the two most prominent second quantization von Neumann algebras are the von Neumann algebras generated by CCR and CAR relations, which provide the mathematical framework for Bosonic and Fermionic interactions.

There are interesting generalizations to these classical constructions in different directions. First to mention are the algebras of field operators on \emph{full} Fock space introduced as a model of free semicircular variables by Voiculescu. These are at the heart of free probability theory (see \cite{VDN92}) and generate von Neumann algebras isomorphic to the free group factors.

In the nineties, M. Bo\.{z}ejko and R. Speicher in \cite{BS94} significantly generalized CCR/CAR algebras by allowing for a much more general type of `symmetry' encoded in a `twist' operator satisfying the Yang--Baxter equation. An important special case are the $q$-Gaussian von Neumann algebras, for which the twist is a scalar multiple $q$ of the tensor flip. Voiculescu's von Neumann algebras generated by free semicircular variables correspond to the case $q=0$ in ths setting.


Fundamental properties like factoriality and amenability of the $q$-Gaussian algebras and more generally von Neumann algebras generated by field operators on a twisted Fock space have been studied extensively in the last two decades. Through the combined efforts of several groups it is now known that these algebras are non-amenable type $\mathrm{II}_1$ factors (see \cite{BKS97, Nou04, Ric05, Kro06, SW18, BKMS21, MS23, Yan23}).


Shlyakhtenko introduced a further generalization of Voiculescu's construction in \cite{Shl97} by incorporating a one-parameter unitary group but keeping the full Fock space. These algebras are the non-tracial counterparts of free group factors and are known as free Araki--Woods factors. They are non-amenable full factors of type $\mathrm{III}$ whenever the associated one-parameter group of unitaries is non-trivial. Free Araki-Woods factors possess many interesting properties: For example, they lack Cartan sub-algebras, have the complete metric approximation property \cite{HR11}, and are strongly solid \cite{BHV18}.

In \cite{Hia03}, Hiai introduced a $q$-deformed version of free Araki--Woods factors, which are non-tracial analogs of the $q$-Gaussian algebras. As in the tracial case, properties like factoriality, non-amenability, etc. are harder to prove for $q\neq 0$ because the powerful tools from free probability do not immediately apply. After a significant number of partial results \cite{Nou04, BM17, SW18, BMRW23}, a complete solution of the factoriality problem was obtained very recently by Kumar, Skalski and Wasilewski \cite{KSW23}.

A further generalization of Hiai's $q$-deformed Araki--Woods algebras, namely mixed $q$-deformed Araki-Woods von Neumann algebras involving multi-scalar symmetry, was introduced in \cite{BKM23} and studied in \cite{Kum23,BKM24}.

In terms of a number of structural properties like factoriality and amenability, these deformed algebras behave either like free group factors or free Araki--Woods factors. However, not much is known about the dependence of the isomorphism class of these second quantization von Neumann algebras on the value of $q$ (or more generally the twist).  On the positive side, it was shown by Guionnet and Shlyakhtenko \cite{GS14} using free monotone transport that for finite-dimensional Hilbert spaces and small $q$ (depending on the dimension of the underlying Hilbert space), the $q$-Gaussian algebras are isomorphic to free group factors. This result was generalized to the non-tracial case of $q$-Araki--Woods factors by \cite{Nel15}.

The range of $q$ for which this result is valid degenerates to $\{0\}$ as the dimension of the underlying Hilbert space goes to infinity, and indeed it was proved by Caspers \cite{Cas23} that infinitely generated $q$-Gaussian von Neumann algebras for $q\neq 0$ do not have Akemann-Ostrand property and hence are not isomorphic to free group factors.

The dependence of the isomorphism type on the dimension of the underlying Hilbert space is even harder -- already for $q=0$ and finite-dimensional Hilbert spaces, it reduces to the famous free group factor isomorphism problem. 

Motivated by applications to quantum field theory, Correa da Silva and Lechner \cite{CL23} (see also \cite{Lec23} for a summary) introduced a general framework for both the tracial and non-tracial second quantization algebras mentioned in the previous paragraphs. Given a Hilbert space $H$, a standard subspace $K\subset H$ in the sense of modular theory, and a bounded self-adjoint Yang-Baxter operator $T$ on $H\otimes H$ satisfying certain positivity condition, they constructed the $T$-twisted Fock space $\mathcal{F}_T(H)$ and von Neumann algebras $\Gamma_T(K\subset H)$, which they call  \emph{twisted Araki--Woods algebras}. These algebras are generated by field operators on $\mathcal F_T(H)$ associated with vectors in $K\subset H$.

Correa da Silva and Lechner showed that if the twist $T$ is compatible with the standard subspace $K\subset H$ in a certain sense and satisfies a condition called crossing symmetry, the vacuum vector in $\mathcal F_T(H)$ is cyclic and separating for $\Gamma_T(K\subset H)$. Depending on the choice of the standard subspace $K\subset H$ and the twist $T$, the twisted Araki--Woods algebras $\Gamma_T(K\subset H)$ take various forms and in particular generalize all the algebras mentioned earlier.  It was proved by Yang \cite{Yan23} that $\Gamma_T(K\subset H)$ is a non-amenable factor when $2\leq\dim(H)<\infty$ and in addition is isomorphic to a free Araki-Woods factor for $\norm{T}$ small (depending on the dimension of $H$).

In this article we introduce operator-valued versions of these twisted Araki--Woods algebras. Operator-valued (free) second quantization algebras were first considered by Shlyakhtenko in his work on operator-valued probability when he studied the von Neumann algebras generated by operator-valued semicircular variables \cite{Shl99} and have found applications in the study of generators of GNS-symmetric quantum Markov semigroups \cite{JRS,Wir22b}. Operator-valued versions of twisted Fock spaces and specifically Bosonic Fock spaces have occurred in the mathematical physics literature in relation to quantum field theory \cite{Ske98,Vas24}.

In our operator-valued setting, the underlying Hilbert space is replaced by a Hilbert bimodule (or correspondence) over a base von Neumann algebra. While it is not obvious how to define a bimodule version of standard subspaces, the equivalent formulation in terms of a one-parameter unitary group and an anti-unitary involution, which was used in Shlyakhtenko's definition of free Araki--Woods algebras, has a bimodule analog in the so-called Tomita correspondences. Tomita correspondences were recently introduced by the second-named author in the study of generators of GNS-symmetric quantum Markov semigroups. In contrast to Shlyakhtenko's von Neumann algebras generated by operator-valued semicircular variables, our approach is thus basis-free. Both approaches are equivalent in the case when the twist is zero and the canonical conditional expectation defined by Shlyakhtenko in this setting is faithful, as we show in this article.

Another difficulty in the operator-valued setting arises if one wants to define analogs of the $q$-Gaussian or $q$-Araki--Woods algebras. The definition of the $q$-deformed tensor product in the scalar-valued case relies on the existence of the tensor flip $\xi\otimes\eta\mapsto \eta\otimes \xi$ on $H\otimes H$ or more generally the action of $S_n$ on $H^{\otimes n}$ that permutes the tensor factors. In the operator-valued case, when the tensor product of Hilbert spaces is replaced by the relative tensor product of correspondences, the naive definition of the tensor flip does not necessarily result in a bounded map, let alone a bimodule map. This difficulty in defining Bosonic or Fermionic Fock bimodules was already noted in the mathematical physics literature \cite{Ske98,Vas24}. For this reason, we work in the framework of general twists, for which the structural requirements in the operator-valued case are transparent. In specific examples, we show that there are still natural operator-valued versions of $q$-Gaussian and $q$-Araki--Woods algebras.

Let us give some more details of our definition of operator-valued twisted Araki--Woods algebras. If $(M,\varphi)$ is a pair consisting of a von Neumann algebra $M$ and a normal semi-finite faithful weight $\phi$ on $M$, a Tomita correspondence over $(M,\varphi)$ is a triple $(\H,\J,(\U_t))$ consisting of a correspondence $\H$ from $M$ to itself together with an anti-unitary involution $\J$ and a strongly continuous one-parameter group of unitaries $(\U_t)$ on $\H$ satisfying some compatibility conditions (see Definition \ref{Def:Tomita correspondence}). Given a Tomita correspondence $(\H,\J,(\U_t))$ over $(M,\varphi)$ and a contractive bimodule map $\T$ on $\H\otimes_\varphi \H$ that satisfies certain positivity condition similar to the ones in \cite{CL23}, we define a twisted inner product on the relative tensor powers $\H^{\otimes_\phi n}$ to construct the twisted Fock bimodule $\F_\T(\H)$. We define the operator-valued twisted Araki--Woods algebra $\Gamma_\T(\H,\J,(\U_t))$ to be the von Neumann algebra generated by the left action of $M$ on $\F_\T(\H)$ and the  field operators $s_\T(\xi):=a_{\T}^{*}(\xi)+a_{\T}(\xi)$ for $\xi$ varying over an appropriate real subspace of $\H$ that depends on $\J$ and $(\U_t)$. The inclusion $M\subset \Gamma_\T(\H,\J,(\U_t))$ comes with a natural faithful normal conditional expectation so that $\varphi$ induces a faithful normal semifinite weight $\hat\phi_\T$ on $\Gamma_\T(\H,\J,(\U_t))$. In the case when $M=\mathbb{C}1$, the algebras $\Gamma_\T(\H,\J,(\U_t))$ reduces to the (scalar-valued) twisted Araki--Woods algebra considered in \cite{CL23}.

It seems more difficult to develop a comprehensive structure theory for the operator-valued twisted Araki--Woods algebras compared to the scalar-valued case for at least two reasons: For one, the structure of the base algebra $M$ comes into play. As a simple observation, when the Tomita correspondence is zero, then the operator-valued twisted Araki--Woods algebra coincides with the base algebra. But even for among non-trivial Tomita correspondences, there are examples for which $\Gamma_\T(\H,\J,(\U_t))\cong M\overline\otimes \Gamma_T(H,J,(U_t))$. As a second reason, while the isomorphism type of a Hilbert space is determined by its dimension, the correspondences over a given von Neumann algebra have a much richer structure. As an example, if one wants to study conjugate variables for the field operators in $\Gamma_\T(\H,\J,(\U_t))$ relative to the base algebra $M$, it follows from Shlyakhtenko's work \cite{Shl00} that the codomain of the derivations should be chosen depending on structure of $\H$.

Our main contributions apart from the definition of the operator-valued twisted Araki--Woods algebras are the following. First, we develop a disintegration theory of Tomita correspondences over semi-finite von Neumann algebras. In the special case of type I factors, this leads to a complete characterization of Tomita correspondences and a characterization of operator-valued twisted Araki--Woods algebras as tensor products of the base algebra with scalar-valued twisted Araki--Woods algebras. In particular, this clarifies the structure of Shlyakhtenko's von Neumann algebras generated by $M$-valued semicircular variables if $M$ is a type I factor. Second, we characterize the modular theory of $\hat\phi_\T$ in terms of the underlying Tomita correspondence. Third, we give sufficient criteria under which the operator-valued twisted Araki--Woods algebras are factors. Furthermore, we also show that several natural von Neumann algebraic constructions fall into the class of operator-valued twisted Araki--Woods algebras.

\subsection*{Outline of the article} In Section \ref{sec:prelim} we gather some preliminary material on weight theory and bimodules over von Neumann algebras. In Section \ref{sec:disintegration_Tomita_corr} we develop a disintegration theory for Tomita correspondence over semi-finite von Neumann algebras, culminating in a complete description of separable Tomita correspondences over type I factors (\Cref{thm:Tomita_corr_type_I}). In Section \ref{sec:twists} we introduce the class of twists we use to define the operator-valued twisted Araki--Woods algebras and discuss several classes of examples. In Section \ref{sec:def_and_mod_theory} we define creation and annihilation operators on the twisted Fock bimodule and introduce the operator-valued twisted Araki--Woods algebra associated with a Tomita correspondence. We further study a natural weight on the operator-valued twisted Araki--Woods algebras and characterize the associated modular conjugation and modular operator (\Cref{thm:mod_theory_twisted}). In Section 6 we discuss three classes of examples for which the operator-valued twisted Araki--Woods algebras coincide with known von Neumann algebraic constructions. In particular, we obtain a complete description of the von Neumann algebras generated by Shlyakhtenko's operator-valued semicircular variables in the case when the base algebra is a type I factor (\Cref{thm:twisted_Araki-Woods_type_I}, \Cref{cor:op-valued_semicircular}). In Section \ref{sec:factoriality} we give two sufficient criteria for the factoriality of operator-valued twisted Araki--Woods algebras (\Cref{prop:mixing_factor}, \Cref{cor:infinite-dim_factor}). Finally, in Section \ref{sec:applications} we discuss some applications to generators of GNS-symmetric quantum Markov semigroups.

\subsection*{Acknowledgments} The authors want to thank the organizers of YMC*A 2023 in Leuven, where this collaboration was conceived. They are grateful to Dan Voiculescu for a valuable historical remark and to Zhiyuan Yang for raising the question if operator-valued weights give rise to Tomita correspondences. M. W. was funded by the Austrian Science Fund (FWF) under the Esprit Programme [ESP 156]. For the purpose of Open Access, the authors have
applied a CC BY public copyright licence to any Author Accepted Manuscript (AAM)
version arising from this submission.

\section{Preliminaries}\label{sec:prelim}

\subsection{Weights on Von Neumann Algebras}

In this section we briefly recall some basic facts and fix the notation regarding weights, modular theory and several classes of bimodules over von Neumann algebras. Except for Subsection \ref{sec:Tomita_prelim}, this material is standard, and we refer to \cite{Tak03} and \cite{Ske01} for more detailed accounts. In the last subsection we recall the definition of Tomita corresponces, which have recently been introduced by the second-named author \cite{Wir22b}.

Let $M$ be a von Neumann algebra. A \emph{weight} on $M$ is a map $\phi\colon M_+\to [0,\infty]$ such that $\phi(x+y)=\phi(x)+\phi(y)$ and $\phi(\lambda x)=\lambda\phi(x)$ for all $x,y\in M_+$ and $\lambda\geq 0$. Here we use the usual convention $0\cdot\infty=0$.

For a weight $\phi$ we define
\begin{align*}
\mathfrak p_\phi&=\{x\in M_+\mid \phi(x)<\infty\},\\
\n_\phi&=\{x\in  M\mid x^\ast x\in \mathfrak p_\phi\},\\
\m_\phi&=\operatorname{lin}\{x^\ast y\mid x,y\in\n_\phi\}.
\end{align*}
The weight $\phi$ is called \emph{normal} if $\sup_j \phi(x_j)=\phi(\sup_j x_j)$ for every bounded increasing net $(x_j)$ in $M_+$, \emph{semi-finite} if $\mathfrak p_\phi$ generates $M$ as a von Neumann algebra, and \emph{faithful} if $\phi(x^\ast x)=0$ implies $x=0$.

Every element of $\m_\phi$ is a linear combination of four elements of $\mathfrak p_\phi$, and $\phi$ can be linearly extended to $\m_\phi$. This extension will still be denoted by $\phi$.

A semi-cyclic representation of $M$ is a triple $(\pi,H,\Lambda)$ consisting of a normal representation of $M$ on $H$ and a $\sigma$-strong$^\ast$ closed linear map $\Lambda$ from a dense left ideal $\n$ of $M$ into $H$ with dense range such that
\begin{align*}
\pi(x)\Lambda(y)&=\Lambda(xy)
\end{align*}
for all $x\in M$ and $y\in\n$. We call $\n$ the \emph{definition ideal} of the semi-cyclic representation $(\pi,H,\Lambda)$.

A normal semi-finite weight $\phi$ on $M$ gives rise to a semi-cyclic representation $(\pi_\phi,L_2(M,\phi),\Lambda_\phi)$ as follows: $L_2(M,\phi)$ is the Hilbert space obtained from $\n_\phi$ after separation and completion with respect to the inner product 
\begin{equation*}
\langle\,\cdot\,,\cdot\,\rangle_\phi\colon\n_\phi\times\n_\phi\to\IC,\,(x,y)\mapsto\phi(x^\ast y),
\end{equation*}
$\Lambda_\phi\colon \n_\phi\to L_2(M,\phi)$ is the quotient map and $\pi_\phi$ is given by $\pi_\phi(x)\Lambda_\phi(y)=\Lambda_\phi(xy)$.

This semi-cyclic representation is essentially uniquely determined by $\phi$ in the following sense: If $(\pi,H,\Lambda)$ is another semi-cyclic representation of $\phi$ with definition ideal $\n_\phi$ and $\langle \Lambda(x),\Lambda(y)\rangle=\phi(x^\ast y)$ for all $x,y\in\n_\phi$, then there exists a unitary $U\colon L_2(M,\phi)\to H$ such that $U\Lambda_\phi=\Lambda$ and $U\pi_\phi(x)U^\ast=\pi(x)$ for all $x\in M$.

The operator 
\begin{equation*}
    \Lambda_\phi(\n_\phi\cap \n_\phi^\ast)\to L^2(M,\phi),\,\Lambda_\phi(x)\mapsto \Lambda_\phi(x^\ast)
\end{equation*}
is an anti-linear closable operator. Let $S_\phi$ denote its closure and $S_\phi=J_\phi\Delta_\phi^{1/2}$ the polar decomposition of $S_\phi$. The operator $J_\phi$, called the \emph{modular conjugation}, is an anti-unitary involution and $\Delta_\phi$, called the \emph{modular operator}, is a positive non-singular self-adjoint operator. We define the right GNS embedding
\begin{equation*}
    \Lambda_\phi^\prime\colon \n_\phi^\ast\to L^2(M,\phi),\,x\mapsto J_\phi\Lambda_\phi(x^\ast)
\end{equation*}
and the right action 
\begin{equation*}
    \pi_\phi^\prime\colon M^{\mathrm{op}}\to B(L^2(M,\phi)),\,x^{\mathrm{op}}\mapsto J_\phi\pi_\phi(x)^\ast J_\phi.
\end{equation*}

The \emph{modular group} associated with $\phi$ is the point-weak$^\ast$ continuous automorphism group $\sigma^\phi$ on $M$ given by $\sigma^\phi_t(x)=\pi_\phi^{-1}(\Delta_\phi^{it}\pi_\phi(x)\Delta_\phi^{-it})$.

\subsection{Von Neumann Bimodules}

Let $A$ be a unital $C^\ast$-algebra. A (right) pre-$C^\ast$-module over $A$ is a right $A$-module $F$ with a sesquilinear map $\langle\cdot|\cdot\rangle_A\colon F\times F\to A$ such that
\begin{itemize}
\item $\langle\xi|\eta\rangle_A x=\langle\xi|\eta x\rangle_A$ for all $\xi,\eta\in F$, $x\in A$,
\item $\langle\xi|\eta\rangle_A ={\langle\eta|\xi \rangle_A}^*$ for all $\xi,\eta\in F$,
\item $\langle\xi|\xi\rangle_A\geq 0$ for all $\xi\in F$,
\item $\langle\xi|\xi\rangle_A=0$ if and only if $\xi=0$.
\end{itemize}
A $C^\ast$-module is a pre-$C^\ast$-module that is complete in the norm $\norm{\langle\cdot|\cdot\rangle_A}^{1/2}$. 

A bounded linear operator $T$ on a $C^\ast$-module $F$ is called \emph{adjointable} if there exists a bounded linear operator $T^\ast$ on $F$ such that
\begin{equation*}
\langle T\xi|\eta\rangle_A=\langle\xi|T^\ast\eta\rangle_A
\end{equation*}
for all $\xi,\eta\in F$. Note that all adjointable operators are right module maps, that is, $T(\xi a)=(T\xi)a$ for all $a\in A$, $\xi\in F$.

Let $A$, $B$ and $C$ be unital $C^\ast$-algebras. A \emph{$C^\ast$ $A$-$B$-module} is a $C^\ast$-module $F$ over $B$ together with a unital $\ast$-homomorphism $\pi$ from $A$ to the adjointable operators on $F$. We simply write $a\xi$ for $\pi(a)\xi$.

In particular, a $C^\ast$ $\IC$-$A$-bimodule is the same as a $C^\ast$ $A$-module and a $C^\ast$ $A$-$\IC$-bimodule the same as a representation of $A$ on a Hilbert space. In the case $A=B$ we simply speak of $C^\ast$ $A$-bimodules.

The \emph{tensor product} $F\overline\odot G$ of a $C^\ast$ $A$-$B$-module $F$ and a $C^\ast$ $B$-$C$-module $G$ is the $C^\ast$ $A$-$C$-module obtained from the algebraic tensor product $F\odot G$ after separation and completion with respect to the $C$-valued inner product given by
\begin{equation*}
\langle\xi\otimes\eta|\xi'\otimes\eta'\rangle_C=\langle\eta|\langle\xi|\xi'\rangle_B\eta'\rangle_C
\end{equation*}
and the actions given by
\begin{equation*}
a(\xi\otimes\eta)=a\xi\otimes \eta,\,(\xi\otimes\eta)c=\xi\otimes\eta c.
\end{equation*}
If $A$ is a $C^\ast$-algebra of bounded operators on the Hilbert space $H$ and $F$ is a $C^\ast$ $A$-module, we can embed $F$ into $B(H,F\overline\odot H)$ by the action 
\begin{equation*}
H\to F\overline\odot H,\,\zeta\mapsto \xi\otimes\zeta
\end{equation*}
for $\xi\in F$. If we refer to the strong operator topology on a $C^\ast$-module in the following, we always mean the strong operator topology in this embedding. If $A$ is von Neumann algebra, we always assume that $A$ is represented on $H$ in standard form.

Let $M$ be a von Neumann algebra on $H$. A \emph{von Neumann $M$-module} is a $C^\ast$ $M$-module $F$ that is strongly closed in $B(H,F\overline\odot H)$. Several equivalent definitions of von Neumann modules have been given in \cite[Proposition 2.9]{Sch02}. In particular, a $C^\ast$-module over a von Neumann algebra is a von Neumann module if and only if it is isometrically isomorphic to a dual space and the right action is weak$^\ast$ continuous. The adjointable operators on a von Neumann $M$-module form a von Neumann algebra $\L_M(F)$.

If $N$ is another von Neumann algebra on $K$, then a $C^\ast$ $M$-$N$-module is a \emph{von Neumann $M$-$N$-module} if it is a von Neumann $N$-module and the left action of $M$ on $F\overline\odot K$ is normal.

\subsection{Correspondences}\label{sec:correspondences}

Let $M$ and $N$ be von Neumann algebras. A \emph{correspondence} from $N$ to $M$ is a Hilbert space $\mathcal H$ together with normal representations $\pi_l^{\H}$ of $M$ and $\pi_r^{\H}$ of $N^{\op}$ on $\H$ such that $\pi_l^\H(M)$ and $\pi_r^{\H}(N^\op)$ commute. With these representations, $\H$ becomes an $M$-$N$-bimodule, and we simply write $x\xi y$ for $\pi_l^\H(x)\pi_r^\H(y^\op)\xi$.

The category of von Neumann $M$-$N$-modules and correspondences from $N$ to $M$ is equivalent in the following sense (see \cite[Section 1.1]{AD90} and \cite[Section 3]{Wir24a} in the $\sigma$-finite case): If $\H$ is a correspondence from $N$ to $M$, then the space $\L(L^2(N)_N,\H_N)$ of all bounded right module maps from $L^2(N)$ to $\H$ becomes a von Neumann $M$-$N$-module when endowed with the actions $x\cdot T\cdot y=\pi_l^\H(x)Ty$ and the $N$-valued inner product $\langle S\vert T\rangle_N=S^\ast T$. The strong operator topology on $\L(L^2(N)_N,\H_N)$ in the sense of the previous subsection coincides with the restriction of the strong operator topology on $B(L^2(N),\H)$.

Conversely, if $F$ is a von Neumann $M$-$N$-module, then the von Neumann $M$-$\IC$-module $F\overline\odot L^2(N)$ becomes a correspondence from $N$ to $M$ when endowed with the right action of $N$ on the tensor factor $L^2(N)$. These two operations are mutually inverse up to unitary equivalence.

If $\phi$ is a normal semi-finite faithful weight on $M$ and $\psi$ is a normal semi-finite faithful weight on $N$, a vector $\xi\in \H$ is called \emph{left $\psi$-bounded} if the map 
\begin{equation*}
    \Lambda_\psi^\prime(\n_\psi\cap \n_\psi^\ast)\to \H,\, \Lambda_\psi^\prime(y)\mapsto \xi y
\end{equation*}
extends to a bounded linear operator $L_\psi(\xi)$ on $L^2(N)$. We write $L^\infty(\H_N,\psi)$ for the set of all left $\phi$-bounded vector in $\H$.

Likewise, a vector $\xi\in\H$ is called \emph{right $\phi$-bounded} if the map
\begin{equation*}
    \Lambda_\phi(\n_\phi\cap\n_\phi^\ast)\to \H,\,\Lambda_\phi(x)\mapsto x\xi
\end{equation*}
extends to a bounded linear operator $R_\phi(\xi)$ on $L^2(M)$. We write $L^\infty(_M\H,\phi)$ for the set of all right $\phi$-bounded vectors in $\H$.

If $\xi\in \H$ is both left $\psi$-bounded and right $\phi$-bounded, it is called $(\phi,\psi)$-bounded. We write $L^\infty(_M\H_N,\phi,\psi)$ for the set of all $(\phi,\psi)$-bounded vectors in $\H$. If the weights $\phi$ and $\psi$ are clear from the context, we simply talk about left-bounded, right-bounded and bounded vectors.

If $\xi\in \H$ is left $\psi$-bounded, then $L_\psi(\xi)$ is a bounded right module map. In fact, right module maps of this form are strongly dense in the set of all bounded right module maps, as the next lemma shows.

\begin{lemma}
The intersection of $\{L_\psi(\xi)\mid \xi\in L^\infty(\H_N,\psi)\}$ with the unit ball of $\L(L^2(N)_N,\H_N)$ is dense in the unit ball of $\L(L^2(N)_N,\H_N)\}$ with respect to the strong operator topology. Moreover, if $\psi$ is finite, then $\{L_\psi(\xi)\mid \xi\in L^\infty(\H_N,\psi)\}=\L(L^2(N)_N,\H_N)$.
\end{lemma}
\begin{proof}
Let $T\in\L(L^2(N)_N,\H_N)\}$ with $\norm{T}\leq 1$. Choose $x_\alpha\in \n_\psi\cap\n_\psi^\ast$ with $\norm{x_\alpha}\leq 1$ and $x_\alpha\to 1$ strongly and let $\xi_\alpha=T\Lambda_\psi(x_\alpha)$. If $y\in \n_\psi\cap \n_\psi^\ast$, then
\begin{equation*}
    \xi_\alpha y=(T\Lambda_\psi(x_\alpha))y=T(\Lambda_\psi(x_\alpha)y)=Tx_\alpha\Lambda_\psi^\prime(y).
\end{equation*}
Thus $\xi_\alpha$ is left-bounded with $L_\psi(\xi_\alpha)=T x_\alpha$. Clearly, $Tx_\alpha\to T$ in the strong operator topology.

If $\psi$ is finite, we can take $x_\alpha=1$ and get $L_\psi(T\Lambda_\psi(1))=T$.
\end{proof}

\subsection{Relative Tensor Product of Correspondences}\label{sec:rel_tensor_prod}

There are three models of the relative tensor product of correspondences, which all come in handy in different situations. We describe here how they relate to each other.

Let $M$, $N$, $Q$ be von Neumann algebras, $\H$ a correspondence from $N$ to $M$ and $\mathcal K$ a correspondence from $Q$ to $N$. The relative tensor product $\H\otimes_N\mathcal K$ of $\H$ and $\mathcal K$ is the correspondence from $Q$ to $M$ with underlying Hilbert space obtained from the algebraic tensor product $\L(L^2(N)_N,\H_N)\odot \mathcal K$ after separation and completion with respect to the inner product
\begin{equation*}
    \langle S\otimes \xi,T\otimes \eta\rangle_{\H\otimes_N \mathcal K}=\langle \xi,\pi_l^{\mathcal K}(S^\ast T)\eta\rangle_{\mathcal K}.
\end{equation*}
Note that if $S,T\in\L(L^2(N)_N,\H_N)$, then $S^\ast T\in (N^\prime)^\prime=N$ so that this expression makes sense. We write $S\otimes_N \xi$ for the image of $S\otimes \xi$ in $\H\otimes_N\mathcal K$. The left and right action are given by
\begin{equation*}
    x(S\otimes_M \xi)y=\pi_l^\H(x)S\otimes_M \pi_r^\H(y^{\mathrm{op}})\eta.
\end{equation*}

The relative tensor product can alternatively be described by the tensor product of the associated von Neumann bimodules in the sense of the previous subsection.

\begin{lemma}
The map 
\begin{equation*}
    \L(L^2(N)_N,\H_N)\overline\odot_N \L(L^2(Q)_Q,\mathcal K_Q)\overline\odot_Q L^2(Q)\to \H\otimes_N \mathcal K,\,S\otimes_N T\otimes_Q\xi\mapsto S\otimes_N T\xi
\end{equation*}
is a unitary bimodule map.
\end{lemma}
\begin{proof}
For $S_1,S_2\in \L(L^2(N)_N,\H_N)$, $T_1,T_2\in \L(L^2(Q)_Q,\mathcal K_Q)$ and $\xi_1,\xi_2\in L^2(Q)$ we have
\begin{align*}
    \langle S_1\otimes_N T_1\xi_1,S_2\otimes_N T_2\xi_2\rangle&=\langle T_1\xi_1,\pi_l^{\mathcal K}(S_1^\ast S_2)T_2\xi_2\rangle\\
    &=\langle\xi_1,T_1^\ast\pi_l^{\mathcal K}(S_1^\ast S_2)T_2\xi_2\rangle\\
    &=\langle\xi_1,\langle S_1\otimes_N T_1\vert S_2\otimes_N T_2\rangle_Q \xi_2\rangle\\
    &=\langle S_1\otimes_N T_1\otimes_Q\xi_1,S_2\otimes_N T_2\otimes_Q \xi_2\rangle.
\end{align*}
Thus the map is isometric. Moreover, the set $\{T\xi\mid T\in \L(L^2(Q)_Q,\mathcal K_Q),\,\xi\in L^2(Q)\}$ is dense in $\mathcal K$ by \cite[Lemma IX.3.3]{Tak03}. Hence the map has dense range. Again, the bimodule property is clear.
\end{proof}

A third description of the relative tensor product relies on the choice of a normal semi-finite faithful weight $\psi$ on $N$. Let $\H\otimes_\psi\mathcal K$ be the Hilbert space obtained from $L^\infty(\H_N,\psi)\odot\mathcal K$ with respect to the inner product
\begin{equation*}
    \langle \xi_1\otimes \eta_1,\xi_2\otimes\eta_2\rangle=\langle \eta_1,\pi_l^{\mathcal K}(L_\psi(\xi_1)^\ast L_\psi(\xi_2))\eta_2\rangle.
\end{equation*}
We write $\xi\otimes_\psi\eta$ for the image of $\xi\otimes\eta$ inside $\H\otimes_\psi\mathcal K$. The Hilbert space $\H\otimes_\psi\mathcal K$ becomes a correspondence from $Q$ to $M$ when endowed with the actions
\begin{equation*}
    x(\xi\otimes_\psi \eta)y=\pi_l^\H(x)\xi\otimes_\psi \pi_r^{\mathcal K}(y^{\mathrm{op}})\eta.
\end{equation*}

\begin{lemma}
The map
\begin{equation*}
    \H\otimes_\psi\mathcal K\to \H\otimes_N \mathcal K,\,\xi\otimes_\psi\eta\mapsto L_\phi(\xi)\otimes_N \eta
\end{equation*}
is a unitary bimodule map.
\end{lemma}
\begin{proof}
For $\xi_1,\xi_2\in \H$ left $\psi$-bounded and $\eta_1,\eta_2\in \mathcal K$ one has
\begin{align*}
    \langle L_\psi(\xi_1)\otimes_N \eta_1,L_\psi(\xi_2)\otimes_N\eta_2\rangle&=\langle \eta_1,\pi_l^{\mathcal K}(L_\psi(\xi_1)^\ast L_\psi(\xi_2))\eta_2\rangle\\
    &= \langle \xi_1\otimes_\psi \eta_1,\xi_2\otimes_\psi\eta_2\rangle
\end{align*}
Thus the map is isometric. As $\{L_\psi(\xi)\mid \xi\in L^\infty(\H_N,\psi)\}$ is dense in $\L(L^2(N)_N,\H_N)$ with respect to the strong operator topology, the map also has dense range. The bimodule property is clear.
\end{proof}

One can also define the relative tensor product of a right and a left module map.

\begin{lemma}\label{lem:rel_tensor_prod_maps}
    Let $M$, $N$, $Q$ be von Neumann algebras, $\H_1$, $\H_2$ correspondences from $N$ to $M$ and $\mathcal K_1,\mathcal K_2$ correspondences from $Q$ to $N$. If $\Phi\colon \H_1\to\H_2$ is a bounded right module map and $\Psi\colon \mathcal K_1\to\mathcal K_2$ is a bounded left module map, then there exists a unique bounded linear operator $\Phi\otimes_N \Psi\colon \H_1\otimes_N\mathcal K_1\to \H_2\otimes_N\mathcal K_2$ such that 
    \begin{equation*}
        (\Phi\otimes_N \Psi)(S\otimes_N \xi)=\Phi S\otimes_N \Psi \xi
    \end{equation*}
    for all $S\in \L(L^2(N)_N,(\H_1)_N)$ and $\xi\in\mathcal K_1$. Moreover, $\norm{\Phi\otimes_N \Psi}\leq \norm{\Phi}\norm{\Psi}$ and $(\Phi\otimes_N\Psi)^\ast=\Phi^\ast\otimes_N\Psi^\ast$.
\end{lemma}
\begin{proof}
    Let $S_1,\dots,S_n\in \L(L^2(N)_N,(\H_1)_N)$ and $\xi_1,\dots,\xi_n\in\mathcal K_1$. Since $\Phi\colon \H_1\to\H_2$ is a bounded right module map, one has $\Phi S_k\in \L(L^2(N)_N,(\H_2)_N)$ for $1\leq k\leq n$.

    By definition of the relative tensor product and the left modularity of $\Psi$,
    \begin{align*}
        \left\lVert\sum_{k=1}^n \Phi S_k\otimes_N \Psi\xi_k\right\rVert^2&=\sum_{j,k=1}^n \langle \Psi\xi_j,(S_j^\ast \Phi^\ast \Phi S_k)\cdot\Psi\xi_k\rangle_{\mathcal K_2}\\
        &\leq \norm{\Phi}^2\sum_{j,k=1}^n \langle \Psi\xi_j,(S_j^\ast S_k)\cdot\Psi\xi_k\rangle_{\mathcal K_2}\\
        &=\norm{\Phi}^2\sum_{j,k=1}^n \langle \Psi\xi_j,\Psi((S_j^\ast S_k)\cdot\xi_k)\rangle_{\mathcal K_2}\\
        &\leq \norm{\Phi}^2\norm{\Psi}^2\sum_{j,k=1}^n \langle \xi_j,(S_j^\ast S_k)\cdot \xi_k\rangle\\
        &=\norm{\Phi}^2\norm{\Psi}^2\left\lVert\sum_{k=1}^n S_k\otimes_N \xi_k\right\rVert^2.
    \end{align*}
    Moreover, if $S_i\in \L(L^2(N)_N,(\H_i)_N)$ and $\xi_i\in \mathcal K_i$ for $i\in\{1,2\}$, then
    \begin{align*}
        \langle \Phi S_1\otimes_N \Psi(\xi_1),S_2\otimes_N \xi_2\rangle
        &=\langle \Psi(\xi_1),(S_1^\ast \Phi^\ast S_2)\cdot\xi_2\rangle\\
        &=\langle\xi_1,(S_1^\ast \Phi^\ast S_2)\cdot\Psi^\ast(\xi_2)\rangle\\
        &=\langle S_1\otimes_N \xi_1,\Phi^\ast S_2\otimes_N \Psi^\ast(\xi_2)\rangle.
    \end{align*}
    Thus $(\Phi\otimes_N\Psi)^\ast=\Phi^\ast\otimes_N\Psi^\ast$.
\end{proof}
It is easy to see that in the other pictures of the relative tensor product, the relative tensor product of maps corresponds to $S\otimes_M T\otimes_M\xi\mapsto \Phi S\otimes_M \Psi T\otimes_M\xi$ and $\xi\otimes_\psi\eta\mapsto \Phi(\xi)\otimes_\psi\Psi(\eta)$, respectively.

\subsection{Tomita Correspondences}\label{sec:Tomita_prelim}

Tomita correspondences were introduced in \cite{Wir22b} in the study of generators of GNS-symmetric quantum Markov semigroups. Let us recall their definition.

\begin{definition}\label{Def:Tomita correspondence}
    Let $M$ be a von Neumann algebra and $\phi$ a normal semi-finite faithful weight on $M$. A \emph{Tomita correspondence} over $(M,\phi)$ is a triple $(\H,\J,(\U_t))$ consisting of a correspondence $\H$ from $M$ to itself, an anti-unitary involution $\J\colon \H\to \H$ and a strongly continuous unitary group $(\U_t)$ on $\H$ such that
    \begin{itemize}
        \item $\J(x\xi y)=y^\ast(\J\xi)x^\ast$ for all $x,y\in M$, $\xi\in \H$,
        \item $\U_t(x\xi y)=\sigma^\phi_t(x)(\U_t\xi)\sigma^\phi_t(y)$ for all $x,y\in M$, $\xi\in\H$,
        \item $\U_t\J=\J\U_t$ for all $t\in\IR$.
    \end{itemize}
\end{definition}
In the context of this article, Tomita correspondences can be seen as bimodule versions of the data underlying Shlyakhtenko's free Araki--Woods algebras and Hiai's $q$-Araki--Woods algebras, namely a complex Hilbert space, an anti-unitary involution and a strongly continuous unitary group that commutes with the involution.

The primary class of examples of Tomita correspondences comes from inclusions of von Neumann algebras with expectation: If $M$ is a von Neumann algebra, $\phi$ a normal semi-finite faithful weight on $M$ and $M\subset \hat M$ is a unital inclusion of von Neumann algebras with a faithful normal expectation $E\colon \hat M\to M$, then $(L^2(\hat M),J_{\phi\circ E},(\Delta_{\phi\circ E}^{it}))$ is Tomita correspondence over $(M,\phi)$. More generally, the same is true if $E$ is replaced by a normal semi-finite faithful operator-valued weight \cite{Haa79}.

A consequence of the construction of the operator-valued free Araki--Woods algebras in \cite[Section 4.2]{Wir22b} (called free Gaussian algebras there) is that in general Tomita correspondences are exactly those subbimodules of the Tomita correspondences from the previous paragraph that are invariant under the modular conjugation and modular group.

Let $(\H,\J,(\U_t))$ be a Tomita correspondence over $(M,\phi)$. By Stone's theorem, there exists a non-singular positive self-adjoint operator $\mathcal{A}$ on $\H$ such that $\U_t=\mathcal{A}^{it}$ for $t\in\IR$. We write $\U_z$ for the (possibly unbounded) operator $\mathcal{A}^{iz}$ with $z\in\IC$. In particular, $\mathcal{A}=\U_{-i}$.

By the spectral theorem, $\bigcap_{z\in\IC}\dom(\U_z)$ is dense in $\H$. In fact, in the context of Tomita correspondences, more is true \cite[Propositon 4.8]{Wir22b}.

\begin{proposition}\label{prop:bdd_analytic_vectors_dense}
    If $(\H,\J,(\U_t))$ is a Tomita correspondence over $(M,\phi)$, then the set
    \begin{equation*}
        \H_0=\left\{\xi\in \bigcap_{z\in\IC}\dom(\U_z): \U_z\xi\in L^\infty(_M\H_M,\phi)\text{ for all }z\in\IC\right\}
    \end{equation*}
    is dense in $\H$.
\end{proposition}

\section{Tomita Correspondences over Type I Factors}\label{sec:disintegration_Tomita_corr}

As discussed in the introduction, the operator-valued free Araki--Woods factors we construct coincide with Shlyakhtenko's free Araki--Woods algebras from \cite{Shl97} in the case when the base algebra is $\IC$. One key tool in the analysis of free Araki--Woods factors is the fact that the Hilbert space can be decomposed into a direct integral of one- and two-dimensional subspaces that are invariant under both the anti-unitary involution and the unitary group.

This decomposition does not immediately carry over to the operator-valued case as one will only get a direct integral of Hilbert spaces that does not respect the bimodule structure. This problem boils down to the fact that the unitary group in the definition of a Tomita correspondence does not consist of bimodule maps, but rather maps that respect the bimodule structure up to a twist by the modular group.

If the underlying von Neumann algebra is semi-finite, this twist can be corrected in a sense, giving rise to a direct integral decomposition of Tomita correspondences analogous to the scalar-valued case that additionally respects the bimodule structure. We combine this with the representation theory of type I factors to obtain a complete characterization of separable Tomita correspondences over type I factors.

Let $(M,\tau)$ be a tracial von Neumann algebra, $h$ a non-singular positive self-adjoint operator affiliated with $M$ and $\phi=\tau(h^{1/2}\,\cdot\,h^{1/2})$. We have $\sigma^\phi_t=h^{it}\,\cdot\,h^{-it}$ and $f(h)$ belongs to the centralizer $M^\phi$ for every bounded Borel function $f\colon\IR\to\IC$.

Let $\H$ be a separable Tomita correspondence from $(M,\phi)$ to itself. By definition, $\U_t\in\pi_l^\H(M^\phi)^\prime\cap \pi_r^\H(M^\phi)^\prime$ for all $t\in\IR$. The following lemma is now immediate.
\begin{lemma}
The operators $\V_t$, $t\in \IR$, given by
\begin{equation*}
\V_t\colon\H\to \H,\,\xi\mapsto\U_t(h^{-it}\xi h^{it})
\end{equation*}
form a strongly continuous unitary group on $\H$.
\end{lemma}

By Stone's theorem, there exists a unique self-adjoint operator $A$ on $\mathcal H$ such that $\V_t=e^{itA}$ for all $t\in\mathbb R$.

\begin{lemma}\label{lem:antisym_J_A}
For every Borel set $B\subset \IR$ one has $\1_{-B}(A)=\J \1_B(A)\J$.
\end{lemma}
\begin{proof}
If $\xi\in \H$ and $t\in\IR$, then
\begin{equation*}
    \J\V_t\J\xi=\J\U_t(h^{-it}(\J\xi)h^{it})=\U_t(h^{-it}\xi h^{it})=\V_t\xi.
\end{equation*}
Thus $\J\V_t\J=\V_t$ for all $t\in\IR$.

Let $f\colon \IR\to\IR$ be a Schwartz function. By the Fourier inversion formula one has
\begin{align*}
    f(-A)
    &=\frac 1{\sqrt{2\pi}}\int_{\IR}\hat f(-t)e^{itA}\,dt\\
    &=\frac 1{\sqrt{2\pi}}\int_{\IR}\hat f(-t)\J e^{itA}\J\,dt\\
    &=\J\left(\frac 1{\sqrt{2\pi}}\int_{\IR}\hat f(t)e^{itA}\,dt\right)\J\\
    &=\J f(A) \J.
\end{align*}
Now the claim follows by approximation of $\1_B$ by Schwartz functions.
\end{proof}

\begin{lemma}
There exists a Borel probability measure $\mu$ on $\IR$ such that $\mu(B)=0$ if and only if $\1_B(A)=0$ and $\mu(B)=\mu(-B)$ for all $B\subset \IR$ Borel.
\end{lemma}
\begin{proof}
Let $(e_j)_{j=0}^\infty$ be an orthonormal basis of $\H$ such that $\J e_j=e_j$ for all $j\geq 0$ and define
\begin{equation*}
    \mu(B)=\sum_{j=0}^\infty 2^{-j}\langle e_j,\1_B(A)e_j\rangle.
\end{equation*}
The property $\mu(B)=0$ if and only if $\1_B(A)=0$ is easy to see, while $\mu(B)=\mu(-B)$ follows from the previous lemma.
\end{proof}
By the direct integral version of the spectral theorem \cite[Theorem 10.9]{Hal13}, the correspondence $\H$ admits a direct integral decomposition $\H= \int_{\IR}^\oplus \H_\omega\,d\mu(\omega)$ such that\begin{equation*}
\V_t=\int_\IR^\oplus e^{i\omega t}\,d\mu(\omega).
\end{equation*}
In particular, the algebra of diagonal operators on $\int_{\IR}^\oplus \H_\omega\,d\mu(\omega)$ is the von Neumann algebra generated by all operators of the form $f(A)$ with $f\colon\IR\to\IR$ bounded Borel.

\begin{lemma}
We have $\V_t\in \pi_l^\H(M)^\prime\cap \pi_r^\H(M)^\prime$ for all $t\in\IR$.
\end{lemma}
\begin{proof}
We only show the result for the left action. The proof for the right action is analogous. For $x\in M$, $\xi\in\H$ and $t\in\IR$ we have
\begin{equation*}
\V_t(x\xi)=\U_t(h^{-it}x\xi h^{it})=\U_t(\sigma^\phi_{-t}(x)h^{-it}\xi h^{it})=x\U_t(h^{-it}\xi h^{it})=x\V_t\xi.\qedhere
\end{equation*}
\end{proof}
By \cite[Corollary IV.8.16]{Tak79} it follows that the left and right action are decomposable representations in the sense that
\begin{align*}
\pi_r^\H(\cdot)&=\int_{\IR}^\oplus \pi_{l,\omega}^\H(\cdot)\,d\mu(\omega),\\
\pi_r^\H(\cdot) &=\int_{\IR}^\oplus \pi_{r,\omega}^\H(\cdot)\,d\mu(\omega),
\end{align*}
where for a.e. $\omega\in\IR$ the map $\pi_{l,\omega}^\H$ (resp. $\pi_{r,\omega}^\H$) is a normal unital $\ast$-homomorphism from $M$ (resp. $M^\op$) to $B(\H_\omega)$ and the images of $\pi_{l,\omega}^\H$ and $\pi_{r,\omega}^\H$ commute . In other words, $\H_\omega$ gets the structure of a correspondence from $M$ to itself and
\begin{align*}
(\U_t\xi)(\omega)=\V_t (\pi_l^\H(h^{it})\pi_r^\H(h^{it})\xi)(\omega)=e^{i\omega t}\pi_{l,\omega}^\H(h^{it})\pi_{r,\omega}^\H(h^{-it})\xi(\omega)
\end{align*}
for every $\xi\in \H$, $t\in \IR$ and a.e. $\omega\in \IR$. Hence $(\U_t)$ decomposes as a direct integral
\begin{equation*}
\U_t=\int_\IR^\oplus e^{i\omega t}\pi_{l,\omega}^\H(h^{it})\pi_{r,\omega}^\H(h^{-it})\,d\mu(\omega).
\end{equation*}

Thus, if we let $\tilde \H_0=\H_0$, $\tilde \H_\omega=\H_\omega\oplus\H_{-\omega}$ for $\omega>0$, we get a direct integral decomposition
\begin{equation*}
\H=\int_{\IR_+}^\oplus \tilde\H_\omega\,d\mu(\omega),
\end{equation*}
where $\IR_+=\{r\in \IR: r\geq 0\}$.
\begin{lemma}
With respect to the direct integral decomposition $\H=\int_{\IR_+}^\oplus \tilde\H_\omega\,d\mu(\omega)$, the operator $\J$ decomposes as $\J=\int_{\IR_+}^\oplus \tilde\J_\omega\,d\mu(\omega)$ such that for a.e. $\omega\in\IR_+$,
\begin{itemize}
    \item $\tilde\J_\omega$ is an anti-unitary involution,
    \item $\tilde \J_\omega(x\xi y)=y^\ast (\tilde\J_\omega\xi)y^\ast$ for all $x,y\in M$, $\xi\in \tilde\H_\omega$,
    \item $\tilde\J_\omega(\H_\omega)=\H_{-\omega}$.
\end{itemize}
\end{lemma}
\begin{proof}
    The algebra of diagonal operators with respect to the direct integral decomposition $\H=\int_{\IR}^\oplus \tilde\H_\omega\,d\mu(\omega)$ coincides with the von Neumann algebra generated by all operators of the form $f(A)$ with $f\colon \IR\to \IR$ an even Borel function. It follows from \Cref{lem:antisym_J_A} that $\J f(A)=f(A)\J$ for every even Borel function $f\colon\IR\to\IR$. Thus $\J$ is decomposable.

    The first two bullet points follow from the corresponding properties of $\J$, while the last is a consequence of $ \J A=-A \J$.
\end{proof}
The last property means that with respect to the direct sum decomposition $\tilde\H_\omega=\H_\omega\oplus\H_{-\omega}$ for $\omega>0$, we can write $\tilde\J_\omega$ as
\begin{equation*}
\tilde\J_\omega=\begin{pmatrix}0&\J_{-\omega}\\\J_\omega&0\end{pmatrix}
\end{equation*}
with mutually inverse anti-unitary maps $\J_\omega\colon \H_\omega\to \H_{-\omega}$, $\J_{-\omega}\colon\H_{-\omega}\to\H_\omega$ that satisfy
\begin{equation*}
    \J_\omega\pi_{l,\omega}^\H(x^\ast)\pi_{r,\omega}^\H(y^{\mathrm{op}})=\pi_{l,-\omega}^\H(y^\ast)\pi_{r,-\omega}^\H(x^{\mathrm{op}})\J_\omega.
\end{equation*}

Now we focus on the case when $M$ is a type I factor. We first recall a well-known result on the structure of correspondences from $M$ to itself.

\begin{lemma}
If $M$ is a type I factor and $\H$ is a correspondence from $M$ to itself, then there exists a Hilbert space $H$ such that $\H\cong L^2(M)\otimes H$ with left and right action on the first tensor factor.
\end{lemma}
\begin{proof}
    Let $M=B(K)$ and $F=\mathcal L(L^2(M)_{M};\H_{M})$. Recall that the trivial bimodule $L^2(B(K))$ can be identified with the space of Hilbert--Schmidt operators on $K$. By \cite[Example 3.3.4]{Ske01} there exists a Hilbert space $H$ such that $F\cong B(K,K\otimes H)$ as von Neumann bimodules over $M$, where the bimodule structure on $B(K,K\otimes H)$ is given by $x\cdot y\cdot z=(x\otimes 1)yz$ for $x,z\in M$, $y\in B(K,K\otimes H)$.
    
    It follows that
    \begin{equation*}
        \H\cong F\overline\odot_{M}L^2(M)\cong \mathrm{HS}(K,K\otimes H)\cong \mathrm{HS}(K)\otimes H\cong L^2(M)\otimes H,
    \end{equation*}
    where $\mathrm{HS}$ stands for the space of Hilbert--Schmidt operators. Here the first isomorphisms follows from the connection between correspondences and von Neumann bimodules discussed in Section \ref{sec:correspondences}, the second one is given by $x\otimes_M y\mapsto xy$, while the third and fourth are obvious.
\end{proof}

Therefore, if $M$ is a type I factor, $\phi$ a normal semi-finite faithful weight on $M$ and $(\H,\J,(\U_t))$ is a Tomita correspondence over $(M,\phi)$, then there exists a measurable field of Hilbert spaces $(H_\omega)_{\omega\in\IR}$ such that $\H_\omega\cong L^2(M)\otimes H_\omega$ as bimodules for a.e. $\omega\in\IR$. Again, we will suppress the isomorphism and simply write $\H_\omega=L^2(M)\otimes H_\omega$. It follows that
\begin{equation*}
\H=L^2(M)\otimes \int_{\IR}^\oplus H_\omega\,d\mu(\omega)
\end{equation*}
with left and right action on the first tensor factor and
\begin{equation*}
\U_t=\Delta_\phi^{it}\otimes\int_\IR^\oplus e^{i\omega t}\,d\mu(\omega).
\end{equation*}
Moreover, if we let $H=\int_{\IR}^\oplus H_\omega\,d\mu(\omega)$ and 
\begin{equation*}
    J\colon H\to H,\,(J\xi)(\omega)=\J_{-\omega}\xi(-\omega),  
\end{equation*}
then $\J=J_\phi\otimes J$.

To summarize, we have obtained the following characterization of separable Tomita correspondences over type I factors.
\begin{theorem}\label{thm:Tomita_corr_type_I}
If $M$ is a type I factor, $\phi$ a normal semi-finite faithful weight on $M$ and $(\H,\J,(\U_t))$ a separable Tomita correspondence over $(M,\phi)$, then there exists a Hilbert space $H$, an anti-unitary involution $J\colon H\to H$, a strongly continuous unitary group $(U_t)$ on $H$ satisfying $[J,U_t]=0$ for all $t\in \IR$ and a unitary bimodule map $V\colon \H\to L^2(M)\otimes H$ such that $V\J=(J_\phi\otimes J)V$ and $V\U_t=(\Delta_\phi^{it}\otimes U_t)V$.
\end{theorem}

\section{Braided Twists}\label{sec:twists}

In this section we discuss the twists we use to define our operator-valued twisted Araki--Woods algebras. It is natural to assume that these twists are bimodule maps, which makes the class of twists more restricted compared to the scalar-valued case. This explains why there seem to be no natural fermionic and bosonic Fock bimodules (or more generally $q$-deformed Fock bimodules) without further assumptions on the underlying correspondence, as has been observed already in the mathematical physics literature \cite{Ske98,Vas24}.

In our conditions on the twists, we do not aim for the highest generality, but rather for a convenient class of twists for which the modular data can be readily computed from the data of the Tomita correspondence and which includes several interesting families of examples.

Assume that $\H$ is a correspondence from $M$ to itself. Given a bounded bimodule map $\T\colon \H\otimes_M\H\to \H\otimes_M\H$, we define the bounded bimodule maps $\T_{k,n}$ and $R_{\T,n}$ on $\H^{\otimes_M n}$ by
\begin{align*}
    \T_{k,n}&=1^{\otimes_M (k-1)}\otimes_M \T\otimes_M 1^{\otimes_M (n-k-1)},\\
    R_{\T,n}&=1+\T_{1,n}+\T_{1,n}\T_{2,n}+\dots+\T_{1,n}\dots\T_{n-1,n}.
\end{align*}
If $n$ is irrelevant in the context, we simply write $\T_k$ instead of $\T_{k,n}$ for $n\geq k+1$.

Moreover, we define inductively,
\begin{align*}
    P_{\T,1}&=1,\\
    P_{\T,n+1}&=(1\otimes_M P_{\T,n})R_{\T,n+1}.
\end{align*}

\begin{definition}
    A \emph{twist} is a contractive bimodule map $\T$ on $\H\otimes_M \H$ such that $P_{\T,n}\geq 0$ for all $n\in\mathbb N$. A \emph{strict twist} is a twist such that $P_{\T,n}$ is additionally invertible for all $n\in\mathbb N$.

    A twist $\T$ is called \emph{braided} if it satisfies the Yang--Baxter equation
    \begin{equation*}
        \T_1\T_2\T_1=\T_2\T_1\T_2.
    \end{equation*}
\end{definition}

\begin{lemma}\label{lem:YBE_twist}
If $\T$ is a self-adjoint contractive bimodule map on $\H\otimes_M \H$ that satisfies the Yang--Baxter equation $\T_1\T_2\T_1=\T_2\T_1\T_2$, then $\T$ is a braided twist.
\end{lemma}
\begin{proof}
This is a special case of \cite[Theorem 2.2]{BS94}. Let $\pi_i\in S_n$ denote the transposition between $i$ and $i+1$ and define a map $\Phi\colon S_n\to B(\H^{\otimes_M n})$ by $\Phi(e)=1$ and $\Phi(\sigma)=\T_{\pi_{i(1)}}\dots\T_{\pi_i(k)}$ whenever $\sigma=\pi_{i(1)}\dots \pi_{i(k)}$ is a reduced word. As explained in \cite[Theorem 1.1]{BS94}, this map is well-defined since $\T$ satisfies the Yang--Baxter equation.

By \cite[Theorem 2.2]{BS94} the operator $\sum_{\sigma\in S_n}\Phi(\sigma)$ is positive.  Moreover, $P_{\T,n}=\sum_{\sigma\in S_n}\Phi(\sigma)$ follows from the Coxeter decomposition of $S_n$ as in the proof of \cite[Theorem 3.1]{BS94}.
\end{proof}

Given a twist $\T$ on $\H\otimes_M \H$, we define a positive semi-definite sesquilinear form on $\H^{\otimes_M n}$ by
\begin{equation*}
    \langle \xi,\eta\rangle_{\T,n}=\langle \xi,P_{\T,n}\eta\rangle.
\end{equation*}
Let $\H_{\T,n}$ denote the Hilbert space obtained from $\H^{\otimes_M n}$ after separation and completion with respect to this inner product. Note that since $P_{\T,n}$ is a bimodule map, $\H_{\T,n}$ inherits the structure of a correspondence from $M$ to itself from $\H^{\otimes_M n}$.

\begin{definition}
The \emph{twisted Fock bimodule} $\F_{\T}(\H)$ over $\H$ is defined as
\begin{equation*}
    \F_{\T}(\H)=L^2(M)\oplus\bigoplus_{n=1}^\infty \H_{\T,n}.
\end{equation*}
\end{definition}
In the case $\T=0$, this is the Hilbert completion of the Fock bimodule over the von Neumann bimodule $\L(L^2(M),\H_M)$, which was introduced by Pimsner \cite{Pim97}.

\begin{remark}
In the scalar-valued case, the most studied examples of twists in this context are the ones that produce the (mixed) $q$-Gaussian algebras. They all involve the flip map on $H\otimes H$. In the operator-valued case, a naive definition of the flip map fails: There is in general no bounded linear map on $\H\otimes_\phi\H$ that maps $\xi\otimes_\phi\eta$ to $\eta\otimes_\phi\xi$ for all bounded vectors $\xi,\eta\in\H$, let alone a bimodule map.

This is one of the reasons why we work with general twists instead of trying to define an operator-valued version of (mixed) $q$-Gaussian algebras. Nevertheless, for specific $\H$ one can still define an analog of the flip map. We will exhibit some instances of such correspondences in the examples below.
\end{remark}

\begin{example}\label{ex:twist_trivial_bimod}
If $\H$ is a direct sum of copies of the trivial bimodule, say $\H=L^2(M)\otimes\ell^2(I)$ with left and right action $x(\xi\otimes e_i)y=x\xi y\otimes e_i$ for the canonical orthonormal basis $(e_i)$ of $\ell^{2}(I)$, we have $\H\otimes_\phi\H\cong L^2(M)\otimes \ell^2(I)\otimes\ell^2(I)$ via
\begin{equation*}
(\Lambda_\phi(x)\otimes e_i)\otimes_\phi (\Lambda_\phi(y)\otimes e_j)\mapsto \Lambda_\phi(xy)\otimes e_i\otimes e_j.
\end{equation*}
If $T$ is a twist on $\ell^2(I)\otimes\ell^2(I)$, let $\T=1_{L^2(M)}\otimes T$. If we identify $\H^{\otimes_\phi n}$ with $L^2(M)\otimes\ell^2(I)^{\otimes n}$ analogous to the case $n=2$, it is easy to see that $\T_{k,n}=1_{L^2(M)}\otimes T_{k,n}$ and $P_{\T,n}=1_{L^2(M)}\otimes P_{T,n}$. Thus $\T$ is a twist, and it is a strict twist if and only if $T$ is a strict twist. Moreover, $\T$ is braided if and only if $T$ is braided.

In particular, if $q_{ij}\in[-1,1]$, we can take
\begin{equation*}
T\colon \ell^2(I)\otimes \ell^2(I)\to \ell^2(I)\otimes \ell^2(I),\,T(e_i\otimes e_j)=q_{ij}e_j\otimes e_i.
\end{equation*}
This is the twist defining the mixed $q$-Gaussian algebras in the scalar-valued case.
\end{example}

\begin{example}\label{ex:twist_group_alg}
Let $H$ be a Hilbert space with anti-unitary involution $J$ and let $\pi$ be a unitary representation of a (discrete countable) group $G$ on $H$ such that $[\pi(g),J]=0$ for all $g\in G$. In other words, $\pi$ is the complexification of an orthogonal action of $G$ on the real Hilbert space $H^J$.

One can turn $\ell^2(G)\otimes H$ into a correspondence $\H_\pi$ from $L(G)$ to itself by defining the left and right action as
\begin{align*}
\lambda_g\cdot(\delta_h\otimes \xi)&=\delta_{gh}\otimes \pi(g)\xi,\\
(\delta_h\otimes \xi)\cdot\lambda_g&=\delta_{hg}\otimes\xi.
\end{align*}
Inductively one can show that 
\begin{equation*}
    V_n\colon \H_\pi^{\otimes_\tau n}\to \ell^2(G)\otimes H^{\otimes n},\,(\delta_{g_1}\otimes\xi_1)\otimes_\tau\dots\otimes_\tau(\delta_{g_n}\otimes \xi_n)\mapsto \delta_{g_1\dots g_n}\otimes \pi(g_1)\xi_2\otimes\dots\otimes \pi(g_1\dots g_{n-1})\xi_n
\end{equation*}
is unitary. Moreover, $V_n$ becomes a bimodule map if $\ell^2(G)\otimes H^{\otimes n}$ is endowed with the left and right action given by
\begin{align*}
\lambda_g\cdot(\delta_h\otimes\xi_1\otimes\dots\otimes\xi_n)&=\delta_{gh}\otimes \pi(g)\xi_1\otimes\dots\pi(g)\xi_n,\\
(\delta_h\otimes\xi_1\otimes\dots\otimes\xi_n)\cdot\lambda_g&=\delta_{hg}\otimes\xi_1\otimes\dots\otimes\xi_n.
\end{align*}
Similar to the previous example, if $(e_i)$ is an orthonormal basis of $H$ and $T$ is the twist on $H\otimes H$ given by $T(e_i\otimes e_j)=q_{ij}e_j\otimes e_i$, then $1_{\ell^2(G)}\otimes T$ is a twist on $\ell^2(G)\otimes H\otimes H\cong \H_\pi^{\otimes_\tau 2}$. More generally, instead of $T$ we can take any twist $S$ on $H\otimes H$ that is equivariant in the sense that $[S,\pi(g)\otimes \pi(g)]=0$ for all $g\in G$, and $1_{\ell^2(G)}\otimes S$ is braided if $S$ is braided.
\end{example} 



\begin{example}\label{ex:twist_centered_bimod}
The center $\mathcal{Z}_M(F)$ of a bimodule $F$ over $M$ is $\mathcal Z_M(F)=\{\xi\in F\mid x\xi=\xi x\text{ for all }x\in M\}$. A von Neumann bimodule $F$ over $M$ is called \emph{centered} if $\mathcal Z_M(F)$ generates $F$ as a von Neumann bimodule (see \cite[Section 3.4]{Ske01}). We also call a correspondence $\H$ from $M$ to itself centered if the von Neumann bimodule $\mathcal L(L^2(M)_M,\H_M)$ is centered. Note that the center of $\mathcal L(L^2(M)_M,\H_M)$ is the set of all bounded bimodule maps from $L^2(M)$ to $\H$.

By \cite[Corollary 3.4.11]{Ske01}, every centered correspondence from $M$ to itself is of the form $\bigoplus_i p_i L^2(M)$ with central projections $p_i\in M$. In particular, if $M$ is a factor, then every centered correspondence from $M$ to itself is a direct sum of copies of the trivial bimodule.

If $F$ is a centered von Neumann bimodule, then there exists a unique unitary bimodule map $\sigma$ on $F\overline\odot_M F$, called the \emph{flip}, such that 
\begin{equation*}
    \sigma(S\otimes T)=T\otimes S
\end{equation*}
for $S,T\in\mathcal Z_M(F)$ (see \cite[Proposition 8.1.1]{Ske01}). 

Let $\H$ be a centered correspondence from $M$ to itself and $F=\L(L^2(M)_M,\H_M)$. In this case, $\sigma\otimes \mathrm{id}_{L^2(M)}$ extends to a unitary bimodule map on $F\overline\odot_M F\overline\odot_M L^2(M)$. Under the canonical identification $F\overline\odot_M F\overline\odot_M L^2(M)\cong \H\otimes_M \H$ discussed in Section \ref{sec:rel_tensor_prod}, one gets a unitary bimodule map $\T$ on $\H\otimes_M\H$, which we also call the flip.

The verification of the Yang--Baxter equation in this case follows along the same lines as in the scalar-valued case. Thus that $q\T$ is a braided twist for $q\in [-1,1]$ by \Cref{lem:YBE_twist}. In the case $q=1$, the twisted Fock bimodule $\F_\T(\H)$ coincides with the Hilbert completion of the symmetric Fock bimodule studied by Skeide (see \cite{Ske98}, \cite[Chapter 8]{Ske01}).
\end{example}

In the case when $(\H,\J,(\U_t))$ is a Tomita correspondence over $(M,\phi)$, we need some more compatibility conditions of the twist. In this case, it is more convenient to work with the weight-dependent formulation of the relative tensor product discussed in Section \ref{sec:rel_tensor_prod}.

With this formulation, it is not hard to see that despite they are not bimodule maps, both $\J$ and $(\U_t)$ can be extended to the relative tensor powers $\H^{\otimes_\phi n}$.

\begin{lemma}
    Let $(\H,\J,(\U_t))$ be a Tomita correspondence over $(M,\phi)$. There exists a unique anti-unitary involution $\J^{(2)}$ and a unique strongly continuous unitary group $(\U_t^{(2)})$ on $\H\otimes_\phi \H$ such that
    \begin{equation*}
        \J^{(2)}(\xi\otimes_\phi\eta)=\J\xi\otimes_\phi\J\eta
    \end{equation*}
    for all left $\phi$-bounded $\xi\in\mathcal H$ and right $\phi$-bounded $\eta\in \H$, and
    \begin{equation*}
        \U_t^{(2)}(\xi\otimes_\phi\eta)=\U_t\xi\otimes_\phi\U_t\eta
    \end{equation*}
    for all left $\phi$-bounded $\xi\in \H$ and $\eta\in\H$.

    Moreover, $(\H\otimes_\phi\H,\J^{(2)},(\U_t^{(2)}))$ is a Tomita correspondence.
\end{lemma}
\begin{proof}
    From the definition of a Tomita correspondence it is easy to see that $\J$ maps left bounded to right bounded vectors, and $R_\phi(\J\xi)=\J L_\phi(\xi)\J$ for left bounded $\xi$ (see \cite[Lemma 4.7]{Wir22b}). Moreover, since right bounded vectors are dense in $\H$ by \cite[Lemma IX.3.3]{Tak03}, the linear span of vectors of the form $\xi\otimes_\phi\eta$ with $\xi$ left bounded and $\eta$ right bounded is dense in $\H\otimes_\phi\H$.

    For $\xi_1,\dots,\xi_n\in L^\infty(\H_M,\phi)$, $\xi_1,\dots,\eta_n\in L^\infty(_M\H,\phi)$ we have
    \begin{align*}
        \left\lVert\sum_{k=1}^n \J\eta_k\otimes_\phi\J\xi_k \right\rVert^2&=\sum_{j,k=1}^n \langle \J\xi_j,L_\phi(\J\eta_j)^\ast L_\phi(\J\eta_k)\cdot \J\xi_k\rangle_\H\\
        &=\sum_{j,k=1}^n \langle \J\xi_j,\J R_\phi(\eta_j)^\ast R_\phi(\eta_k)\J\cdot \J\xi_k\rangle_\H\\
        &=\sum_{j,k=1}^n \langle\J\xi_j,\J(\xi_k\cdot \J R_\phi(\eta_k)^\ast R_\phi(\eta_j)\J)\rangle_\H\\
        &=\sum_{j,k=1}^n \langle \xi_k\cdot \J R_\phi(\eta_k)^\ast R_\phi(\eta_j)\J,\xi_j\rangle_\H.
    \end{align*}
    By \cite[Proposition IX.3.15]{Tak03}, the last line equals $\sum_{j,k}\langle \xi_k\otimes \eta_k,\xi_j\otimes \eta_j\rangle_\H$. Thus $\J^{(2)}$ is isometric. 

    For $\U_t^{(2)}$ we have
    \begin{align*}
        \left\lVert\sum_{k=1}^n \U_t\xi_j\otimes\U_t\eta_k\right\rVert^2
        &=\sum_{j,k=1}^n \langle \U_t\eta_j,L_\phi(\U_t \xi_j)^\ast L_\phi(\U_t\xi_k)\cdot\U_t\eta_k\rangle_\H\\
        &=\sum_{j,k=1}^n \langle \U_t \eta_j, \sigma_\phi^{it}(L_\phi(\xi_j)^\ast L_\phi(\xi_k))\cdot \U_t \eta_k\rangle_\H\\
        &=\sum_{j,k=1}^n \langle \U_t \eta_j,\U_t(L_\phi(\xi_j)^\ast L_\phi(\xi_k)\cdot \eta_k)\rangle_\H\\
        &=\sum_{j,k=1}^n \langle \eta_j,L_\phi(\xi_j)^\ast L_\phi(\xi_k)\cdot\eta_k\rangle_\H\\
        &=\left\lVert \sum_{k=1}^n \xi_k\otimes\eta_k\right\rVert^2,
    \end{align*}
where we used the identity $L_\phi(\U_t\xi_j)^\ast L_\phi(U_t\xi_k)=\sigma^\phi_t(L_\phi(\xi_j)^\ast L_\phi(\xi_k))$ from \cite[Lemma 4.7]{Wir22b} for the second equality. Thus $\U_t^{(2)}$ is isometric. The remaining claims are easy to check.
\end{proof}

As a consequence, we can inductively define $\J^{(n)}$ and $(\U_t^{(n)})$ on the relative tensor power $\H^{\otimes_\phi n}$ such that $(\H^{(n)},\J^{(n)},(\U_t^{(n)}))$ becomes a Tomita correspondence. If $\xi_1,\dots,\xi\in \H$ are bounded vectors, then
\begin{align*}
    \U_t^{(n)}(\xi_1\otimes_\phi\dots\otimes_\phi\xi_n)&=\U_t\xi_1\otimes_\phi\dots\otimes_\phi\U_t\xi_n,\\
    \J^{(n)}(\xi_1\otimes_\phi\dots\otimes_\phi\xi_n)&=\J\xi_n\otimes_\phi\dots\otimes_\phi\J\xi_1.
\end{align*}

For a twist $\T$ on $\H\otimes_\phi\H$ let $\tilde R_{\T,n}=1+\T_{n-1}+\T_{n-1}\T_{n-2}+\dots+\T_{n-1}\dots \T_1$. The following result is a consequence of the Coxeter representation of the symmetric group with transpositions as generators. The proof of \cite[Lemma 3.16]{CL23} carries over to our setting.

\begin{lemma}\label{lem:inn_prod_braided_twist}
    If $\T$ is a braided twist on $\H\otimes_\phi\H$, then
    \begin{equation*}
        P_{\T,n+1}=(P_{\T,n}\otimes_\phi 1)\tilde R_{\T,n+1}
    \end{equation*}
    for all $n\in\mathbb N$.
\end{lemma}

\begin{lemma}
    \begin{enumerate}[(a)]
        \item If $[\T,\J^{(2)}]=0$, then $\J^{(n)}\T_k=\T_{n-k}\J^{(n)}$ for all $n\in\mathbb N$, $k\leq n-1$. In particular, $\J^{(n)}\tilde R_{\T,n}= R_{\T,n}\J^{(n)}$ for all $n\in\mathbb N$.
        \item If additionally $\T$ is braided, then $[P_{\T,n},\J^{(n)}]=0$ for all $n\in\mathbb N$.
    \end{enumerate}
\end{lemma}
\begin{proof}
    \begin{enumerate}[(a)]
        \item If $\xi_1,\dots,\xi_n\in\H$ are bounded, then
        \begin{align*}
            \J^{(n)}\T_k(\xi_1\otimes_\phi\dots\otimes_\phi\xi_n)&=\J^{(n)}(\xi_1\otimes_\phi\dots\otimes_\phi\T(\xi_k\otimes_\phi\xi_{k+1})\otimes_\phi\dots\otimes_\phi\xi_n)\\
            &=\J\xi_n\otimes_\phi\dots\otimes_\phi\J^{(2)}\T(\xi_k\otimes_\phi\xi_{k+1})\otimes_\phi\dots\otimes_\phi \J\xi_1\\
            &=\J\xi_n\otimes_\phi\dots\otimes_\phi\T\J^{(2)}(\xi_k\otimes_\phi\xi_{k+1})\otimes_\phi\dots\otimes_\phi \J\xi_1\\
            &=\T_{n-k}\J^{(n)}(\xi_1\otimes_\phi\dots\otimes_\phi\xi_n).
        \end{align*}
        Now the first claim follows from the density of bounded vectors in $\H$. The second claim is then immediate from the definitions of $R_{\T,n}$ and $\tilde R_{\T,n}$.
        \item We proceed by induction. Assume that the claim is true for $n\in\mathbb N$. If $\xi\in\H^{\otimes_\phi n}$ and $\eta\in \H$ are bounded, then
        \begin{align*}
            (1\otimes_\phi P_{\T,n})\J^{(n+1)}(\xi\otimes_\phi \eta)
            &=(1\otimes_\phi P_{\T,n})(\J\eta\otimes_\phi \J^{(n)}\xi)\\
            &=\J\eta\otimes_\phi P_{\T,n}\J^{(n)}\xi\\
            &=\J\eta\otimes_\phi \J^{(n)}P_{\T,n}\xi\\
            &=\J^{(n+1)}(P_{\T,n}\xi\otimes_\phi \eta)\\
            &=\J^{(n+1)}(P_{\T,n}\otimes_\phi 1)(\xi\otimes_\phi \eta).
        \end{align*}
        Hence $(1\otimes_\phi P_{\T,n})\J^{(n+1)}=\J^{(n+1)}(P_{\T,n}\otimes_\phi 1)$.

        Now we deduce from (a) that
        \begin{align*}
            P_{\T,n+1}\J^{(n+1)}
            &=(1\otimes_\phi P_{\T,n})R_{\T,n+1}\J^{(n+1)}\\
            &=(1\otimes_\phi P_{\T,n})\J^{(n+1)}\tilde R_{\T,n+1}\\
            &=\J^{(n+1)}(P_{\T,n}\otimes_\phi 1)\tilde R_{\T,n+1}.
        \end{align*}
        Since $\T$ is braided, the last expression coincides with $\J^{(n+1)}P_{\T,n+1}$ by \Cref{lem:inn_prod_braided_twist}.\qedhere
    \end{enumerate}
\end{proof}

The analogous result for the unitary group $(\U_t)$ is immediate.
\begin{lemma}
    If $\T$ is a twist such that $[\T,\U_t^{(2)}]=0$ for all $t\in\mathbb R$, then $[P_{\T,n},\U_t]=0$ for all $t\in\mathbb R$ and $n\in \mathbb N$.
\end{lemma}

\begin{definition}
If $(\H,\J,(\U_t))$ is a Tomita correspondence, we say that a braided twist $\T$ on $\H\otimes_\phi\H$ is \emph{compatible} if $[\T,\J^{(2)}]=0$ and $[\T,\U_t^{(2)}]=0$ for all $t\in\IR$.   
\end{definition}

If $\T$ is a compatible twist, it follows from the previous two lemmas that $\J^{(n)}$ and $\U_t^{(n)}$ leave $\ker P_{\T,n}$ invariant and are isometric with respect to $\langle \,\cdot\,,\cdot\,\rangle_{\T,n}$ for all $n\in\mathbb N$. Thus they can be continuously extended to (anti-) unitary operators on $\H_{\T,n}$. It is then not hard to see that these extensions make $\H_{\T,n}$ and $\F_\T(\H)$ into Tomita correspondences.

We write $\F_\T(\J)$ and $\F_\T(\U_t)$ for the bounded operators on $\F_\T(\H)$ that act on $\H_{\T,n}$ by $\J^{(n)}$ and $\U_t^{(n)}$, respectively.

Let us discuss under which conditions the twists in the examples from the beginning of the section are compatible twists.

\begin{example}\label{Ex:twist_trivial_bimod}
As in Example \ref{ex:twist_trivial_bimod}, let $\H=L^2(M)\otimes \ell^2(I)$ with left and right action $x(\xi\otimes e_i)y=x\xi y\otimes e_i$, let $T$ be a braided twist on $\ell^2(I)\otimes\ell^2(I)$ and let $\T=1_{L^2(M)}\otimes T$.

If $\J=J_\phi\otimes J$, where $J$ is the complex conjugation on $\ell^2(I)$, and $\U_t=\Delta_\phi^{it}\otimes U_t$ for some strongly unitary group $(U_t)$ on $\ell^2(I)$ such that $[U_t,J]=0$ for all $t\in\IR$, then 
\begin{align*}
    [\T,\U_t^{(2)}]&=\Delta_\phi^{it}\otimes [T,U_t^{(2)}],\\
    [\T,\J^{(2)}]&=J_\phi\otimes [T,J^{(2)}].
\end{align*}
Hence $\T$ is compatible if and only if $T$ is compatible.
\end{example}

\begin{example}\label{Ex:twist_group_alg}
As in Example \ref{ex:twist_group_alg}, let $G$ be a discrete group, $\pi$ a unitary representation of $G$ on a Hilbert space $H$ and $J$ an anti-unitary involution on $H$ such that $[\pi(g),J]=0$ for all $g\in G$. As discussed in \Cref{ex:twist_trivial_bimod}, we can turn $\ell^2(G)\otimes H$ into a correspondence $\H_\pi$ from $L(G)$ to itself. Let $T$ be an equivariant braided twist on $H\otimes H$ and $\T=1_{\ell^2(G)}\otimes T$.

On $\H_\pi$ one can define an anti-unitary involution $\J_\pi$ by $\J_\pi(\delta_g\otimes \xi)=\delta_{g^{-1}}\otimes \pi(g^{-1})J\xi$. If $(U_t)$ is an equivariant strongly continuous unitary group on $H$ such that $[U_t,J]=0$ for all $t\in\IR$, let $\U_t=1_{\ell^2(G)}\otimes U_t$. Then $(\H_\pi,\J_\pi,(\U_t))$ is a Tomita correspondence over $(L(G),\tau)$. Moreover, as in the previous example, one can show that $\T$ is compatible if and only if $T$ is compatible.
\end{example}

\begin{example}\label{Ex:Twist_Centered_Bimod}
Let $(\H,\J,(\U_t))$ be a Tomita correspondence over $(M,\phi)$ such that $\H$ is centered in the sense discussed in \Cref{ex:twist_centered_bimod} and let $\T$ be the flip map on $\H\otimes_\phi\H$. We write $F$ for $\L(L^2(M)_M,\H_M)$.

Let 
\begin{equation*}
    \mathfrak a_0=\left\{x\in\n_\phi\cap \n_\phi^\ast\cap\bigcap_{z\in\mathbb C}\dom(\sigma^\phi_z)\,\bigg\vert\, \sigma^\phi_{z}(x)\in \n_\phi\cap \n_\phi^\ast\text{ for all }z\in\IC\right\}.
\end{equation*}
By standard approximation techniques \cite[Theorems VI.2.2, VII.2.6]{Tak03}, $\mathfrak a_0$ is a strongly$^\ast$ dense $\ast$-subalgebra of $M$. Since $F$ is centered, it follows that the linear span of $\{Rx\mid R\in\mathcal Z(F),\,x\in\mathfrak a_0\}$ is strongly dense in $F$. Together with \cite[Lemma IX.3.3]{Tak03}, this implies that the linear span of $\{R\Lambda_\phi(x)\mid R\in \mathcal Z(F),x\in\mathfrak a_0\}$ is dense in $\H$.

With the canonical isomorphism between $F\overline\odot_M F\overline\odot L^2(M)$ and $\H\otimes_\phi\H$ discussed in Section \ref{sec:rel_tensor_prod}, it is not hard to see that 
\begin{equation*}
    \T(R_1\Lambda_\phi(x_1)\otimes_\phi R_2\Lambda_\phi(x_2))=R_2\Lambda_\phi(x_1)\otimes_\phi R_1\Lambda_\phi(x_1)
\end{equation*}
for $R_1,R_2\in \mathcal Z(F)$, $x_1,x_2\in\mathfrak a_0$.

If $R\in \mathcal Z(F)$, that is, $R$ is a bounded bimodule map from $L^2(M)$ to $\H$, then $\J R J_\phi$ and $\U_t R \Delta_\phi^{-it}$ are also bounded bimodule maps. Thus
\begin{align*}
    \J^{(2)}\T(R_1\Lambda_\phi(x_1)\otimes_\phi R_2\Lambda_\phi(x_2))&=\J^{(2)}(R_2\Lambda_\phi(x_1)\otimes_\phi R_1\Lambda_\phi(x_2))\\
    &=\J R_1\Lambda_\phi(x_2)\otimes_\phi\J R_2\Lambda_\phi(x_1)\\
    &=\J R_1 J_\phi J_\phi\Lambda_\phi(x_2)\otimes_\phi\J R_2 J_\phi J_\phi \Lambda_\phi(x_1)\\
    &=\J R_1 J_\phi \Lambda_\phi(\sigma^\phi_{i/2}(x_2)^\ast)\otimes_\phi \J R_2 J_\phi \Lambda_\phi(\sigma^\phi_{i/2}(x_1)^\ast)\\
    &=\T(\J R_2 J_\phi\Lambda_\phi(\sigma^\phi_{i/2}(x_2)^\ast)\otimes_\phi \J R_1 J_\phi\Lambda_\phi(\sigma^\phi_{i/2}(x_1)^\ast))\\
    &=\T(\J R_2 \Lambda_\phi(x_2)\otimes_\phi \J R_1 \Lambda_\phi(x_1))\\
    &=\T \J^{(2)}(R_1\Lambda_\phi(x_1)\otimes_\phi R_2\Lambda_\phi(x_2))
\end{align*}
and
\begin{align*}
    \U_t^{(2)}\T(R_1\Lambda_\phi(x_1)\otimes_\phi R_2\Lambda_\phi(x_2))&=\U_t^{(2)}(R_2\Lambda_\phi(x_1)\otimes_\phi R_1\Lambda_\phi(x_2))\\
    &=\U_t R_2\Delta_\phi^{-it}\Lambda_\phi(\sigma^\phi_t(x_1))\otimes_\phi\U_t R_1 \Delta_\phi^{-it}\Lambda_\phi(\sigma^\phi_t(x_2))\\
    &=\T(\U_t R_1 \Delta_\phi^{-it}\Lambda_\phi(\sigma^\phi_t(x_1))\otimes_\phi \U_t R_2\Delta_\phi^{-it}\Lambda_\phi(\sigma^\phi_t(x_2)))\\
    &=\T(\U_t R_1\Lambda_\phi(x_1)\otimes_\phi \U_t R_2\Lambda_\phi(x_2))\\
    &=\T\U_t^{(2)}(R_1\Lambda_\phi(x_1)\otimes_\phi R_2 \Lambda_\phi(x_2)).
\end{align*}
for all $R_1,R_2\in\mathcal Z(F)$ and $x_1,x_2\in\mathfrak a_0$. Together with the density properties discussed above, this implies that $[\J^{(2)},\T]=0$ and $[\U_t^{(2)},\T]=0$. Therefore $\T$ (and thus also $q\T$ for $q\in [-1,1]$) is a compatible twist.
\end{example}

\section{Operator-Valued Twisted Araki--Woods Algebras}\label{sec:def_and_mod_theory}

In this section we introduce the operator-valued twisted Araki--Woods algebras as von Neumann algebras on the twisted Fock space generated by the left action of the base algebra and certain operator-valued analogs of the field operators. In the second part, we show that these operator-valued twisted Araki--Woods algebras carry a natural ``vacuum'' weight and show how under a suitable condition on the twist, the associated modular data can be described in terms of the underlying Tomita correspondence.

\subsection{Twisted Creation and Annihilation Operators}

We start with the definition of the creation operators. To define them in the twisted case, we need some operator inequalities, which will also be useful later for factoriality criteria in Section \ref{sec:factoriality}. The proofs are essentially the same as in the scalar-valued case (see \cite{BS94,Boz98}). In the following, $M$ is a von Neumann algebra, $\phi$ is a normal semi-finite faithful weight on $M$ and $(\H,\J,(\U_t))$ is a Tomita correspondence over $(M,\phi)$.

\begin{lemma}\label{lem:bound_twisted_norm_tensor_product}
    If $\T$ is a twist on $\H\otimes_\phi\H$, then
    \begin{equation*}
        c_n(\norm{\T})(1\otimes_\phi P_{\T,n})\leq P_{\T,n+1}\leq \left(\sum_{k=0}^n \norm{\T}^k\right)(1\otimes_\phi P_{\T,n}),
    \end{equation*}
    where $c_n(q)=\prod_{k=1}^n (1+q^k)^{-1}\prod_{k=1}^{n}(1-q^k)$.

    If $\T$ is a braided twist, the inequalities remain true if $1\otimes_\phi P_{\T,n}$ is replaced by $P_{\T,n}\otimes_\phi 1$.
\end{lemma}
\begin{proof}
    The second inequality follows from the same arguments as the one used in \cite[Theorem 3.1]{BS94} in the scalar-valued case: Let $d_n=\left(\sum_{k=0}^n \norm{\T}^k\right)$. Since $P_{\T,n}\geq 0$, we have
    \begin{align*}
        P_{\T,n+1}^2
        =(1\otimes_\phi P_{\T,n})R_{\T,n+1}R_{\T,n+1}^\ast(1\otimes_\phi P_{\T,n})
        \leq \norm{R_{\T,n+1}}^2(1\otimes_\phi P_{\T,n})^2.
    \end{align*}
    By the definition of $R_{T,n+1}$ and \Cref{lem:rel_tensor_prod_maps}, we have $\norm{R_{\T,n+1}}\leq d_n$. Thus $P_{\T,n+1}\leq d_n (1\otimes_\phi P_{\T,n})$.

    For the first inequality, we follow the argument from \cite[Theorem 6]{Boz98} in the scalar-valued case. The inequality holds trivially when $\norm{\T}=1$ since $P_{\T,n}\geq 0$ for all $n\in \mathrm{N}$ by the definition of twist. Thus, we consider only the case when $\norm{\T}<1$. It follows by \Cref{lem:rel_tensor_prod_maps}, the lines of \cite[Lemma 4(b)]{Boz98} that
    \begin{equation*}
      R_{\T,n+1}=\prod_{k=1}^{n-1}(1-\T_n^2\T_{n-1}\cdots \T_k)(1+\T_n) \prod_{l=n}^{1}(1-\T_n\cdots\T_l)^{-1},
    \end{equation*}
    which implies
    \begin{equation*}
      R_{\T,n+1}^{-1}= \prod_{l=1}^{n}(1-\T_n\cdots\T_l)(1+\T_n)^{-1} \prod_{k=n-1}^{1}(1-\T_n^2\T_{n-1}\cdots \T_k)^{-1}.
    \end{equation*}
    Therefore,
    \begin{align*}
    \norm{R_{\T,n+1}^{-1}}&\leq \prod_{k=1}^{n}(1+\norm{\T}^k) (1+\norm{\T})^{-1}\prod_{k=1}^{n-1} (1+\norm{\T}^{n-k+1})^{-1}\\
    &=(1+\norm{\T})^{-1} \prod_{k=1}^{n}(1+\norm{\T}^k)\prod_{k=2}^{n} (1+\norm{\T}^{k})^{-1}\\
    &\leq(1-\norm{\T})^{-1} \prod_{k=1}^{n}(1+\norm{\T}^k)\prod_{k=2}^{n} (1-\norm{\T}^{k})^{-1}.
    \end{align*}
Hence, $\norm{R_{\T,n+1}^{-1}}\leq c_n(\norm{\T})^{-1}$. This implies that $\norm{(R_{\T,n+1}^{-1})^\ast(R_{\T,n+1}^{-1})}\leq (c_n(\norm{\T})^{-1})^2$ and hence $R_{\T,n+1}R_{\T,n+1}^\ast\geq c_n(\norm{\T})^2 $. Now it follows from the previous calculation that
    \begin{align*}
        P_{\T,n+1}^2
        =(1\otimes_\phi P_{\T,n})R_{\T,n+1}R_{\T,n+1}^\ast(1\otimes_\phi P_{\T,n})
        \geq c_n(\norm{\T})^2 (1\otimes_\phi P_{\T,n})(1\otimes_\phi P_{\T,n}).
    \end{align*}
    If $\T$ is a braided twist, then $P_{\T,n+1}=(P_{\T,n}\otimes_\phi 1)\tilde R_{\T,n+1}$ by \Cref{lem:inn_prod_braided_twist}, and the proof of the inequalities with $1\otimes_\phi P_{\T,n}$ replaced by $P_{\T,n}\otimes_\phi 1$ is analogous.
    \end{proof}
\begin{lemma}\label{lem:norm_creation_op}
    Let $\T$ be a twist on $\H\otimes_\phi\H$. If $\xi\in \H$ is left bounded and $\eta\in \H^{\otimes_\phi n}$,then
    \begin{equation*}
        \langle \xi\otimes_\phi \eta,P_{\T,n+1}(\xi\otimes_\phi \eta)\rangle_{\H^{\otimes_\phi(n+1)}}\leq \left(\sum_{k=0}^n \norm{\T}^k\right)\norm{L_\phi(\xi)}^2\norm{\eta}_{\T,n}^2.
    \end{equation*}
\end{lemma}
\begin{proof}
    Let $d_n=\left(\sum_{k=0}^n \norm{\T}^k\right)$. By \Cref{lem:bound_twisted_norm_tensor_product} we have $P_{\T,n+1}\leq d_n (1\otimes_\phi P_{\T,n})$.

    It follows that
    \begin{align*}
     \langle \xi\otimes_\phi \eta,P_{\T,n+1}(\xi\otimes_\phi \eta)\rangle_{\H^{\otimes_\phi(n+1)}}
     &\leq d_n\langle \xi\otimes_\phi \eta,\xi\otimes_\phi P_{\T,n}\eta\rangle_{\H^{\otimes_\phi (n+1)}}\\
     &=d_n\langle \eta,L_\phi(\xi)^\ast L_\phi(\xi)\cdot P_{\T,n}\eta\rangle_{\H^{\otimes_\phi n}}\\
     &=d_n\langle P_{\T,n}^{1/2}\eta,L_\phi(\xi)^\ast L_\phi(\xi)\cdot P_{\T,n}^{1/2}\eta\rangle_{\H^{\otimes_\phi n}}\\
     &\leq d_n \norm{L_\phi(\xi)}^2\langle \eta,P_{\T,n}\eta\rangle_{\H^{\otimes_\phi n}}\\
     &=d_n \norm{L_\phi(\xi)}^2\norm{\eta}_{\T,n}^2,
    \end{align*}
    where we used that $P_{\T,n}$ is positive and a bimodule map.
\end{proof}

As a consequence of the previous lemma, for every left bounded $\xi\in \H$, the map
\begin{equation*}
    \H^{\otimes_\phi n}/\ker P_{\T,n}\to \H^{\otimes_\phi n+1}/\ker P_{\T,n+1},\,\eta+\ker P_{\T,n}\mapsto \xi\otimes_\phi\eta+\ker P_{\T,n+1}
\end{equation*}
extends to a bounded linear operator $a^\ast_{\T,n}(\xi)$ from $\H_{\T,n}$ to $\H_{\T,n+1}$, called (twisted) \emph{creation operator}. We further set $a_{\T,0}^\ast(\xi)=L_\phi(\xi)$ as operator from $L^2(M)$ to $\H$.

If $\norm{\T}<1$, then
\begin{equation*}
    \norm{a^\ast_{\T,n}(\xi)}\leq \sqrt{d_n}\norm{L_\phi(\xi)}\leq \frac{\norm{L_\phi(\xi)}}{\sqrt{1-\norm{T}}}
\end{equation*}
is bounded for $n\in\mathbb N$.

We define the (twisted) \emph{annihilation operator} by $a_{\T,n+1}(\xi)=a^\ast_{\T,n}(\xi)^\ast$ and $a_{\T,0}(\xi)=0$. If $\norm{\T}<1$, then the previous estimate shows that the twisted creation and annihilation operator can be extended to bounded operators on the twisted Fock bimodule, which we denote by $a_\T^\ast(\xi)$ and $a_\T(\xi)$, respectively.

In general however, as in the case of the scalar-valued Bosonic Fock space, one has to deal with unbounded operators. We show in the next lemma that one has unique self-adjoint realizations of the field operators even in the unbounded case.

\begin{lemma}\label{lem:field_op_ess_self-adj}
For every left-bounded vector $\xi\in \H$ there exists a unique self-adjoint operator $s_\T(\xi)$ on $\F_\T(\H)$ such that $L^2(M)\subset \dom(s_\T(\xi))$, $s_\T(\xi)\eta=a_{\T,0}^\ast(\xi)\eta$ for $\eta\in L^2(M)$ and $\H_{\T,n}\subset \dom(s_\T(\xi))$ for all $n\in\mathbb N$ and $s_\T(\xi)\eta=(a_{\T,n}^\ast(\xi)+a_{\T,n}(\xi))\eta$ for $\eta\in\H_{\T,n}$.
\end{lemma}
\begin{proof}
For $\xi\in \H$ left-bounded let $a_\T^\ast(\xi)$ and $a_\T(\xi)$ denote the operators whose domain contains all finitely non-zero sequences in $\F_\T(\H)$ and which coincide with $a_{\T,n}^\ast(\xi)$ and $a_{\T,n}(\xi)$, respectively, on $\H_{\T,n}$. Since $a_{\T}^\ast(\xi)$ and $a_\T(\xi)$ are mutually adjoint on their domain, the sum $s_\T^0(\xi)=a_\T^\ast(\xi)+a_\T(\xi)$ is symmetric and densely defined. We show that it is essentially self-adjoint using Nelson's theorem on analytic vectors \cite[Lemma 5.1]{Nel59}.

Let $\H_{\T,\leq n}=L^2(M)\oplus\bigoplus_{j=1}^n \H_{\T,j}$. By the previous lemma,
\begin{align*}
    \norm{a_\T^\ast(\xi)\eta}_{\T}\leq \left(\sum_{k=0}^n\norm{\T}^k\right)^{1/2}\norm{L_\phi(\xi)}\norm{\eta}_\T\leq \sqrt{n+1}\norm{L_\phi(\xi)}\norm{\eta}_\T\leq (n+1)\norm{L_\phi(\xi)}\norm{\eta}_{\T}
\end{align*}
for all $\eta\in \H_{\T,\leq n}$. Likewise, $\norm{a_\T(\xi)\eta}_\T\leq  n \norm{L_\phi(\xi)}\norm{\eta}_\T$ for all $\eta\in \H_{\T,\leq n}$.

If $\eta\in \H_{\T,\leq n}$, then $s_\T^0(\xi)^k\eta\in \H_{\T,\leq n+k}$. As
\begin{align*}
    \norm{s_\T^0(\xi)^{k+1}\eta}_\T&=\norm{(a_\T^\ast(\xi)+a_\T(\xi))s_\T^0(\xi)^k\eta}\leq 2(n+k+1)\norm{L_\phi(\xi)}\norm{s_\T^0(\xi)^k\eta}_\T,
\end{align*}
it follows by induction that 
\begin{equation*}
    \norm{s_\T^0(\xi)^k \eta}_\T\leq \frac{2^k(n+k)!}{n!}\norm{L_\phi(\xi)}^k \norm{\eta}_\T.
\end{equation*}
Thus
\begin{equation*}
    \sum_{k=0}^\infty \frac{\norm{s_\T^0(\xi)^k\eta}_\T}{k!}s^k\leq \norm{\eta}_\T\sum_{k=0}^\infty \binom{n+k}{n}(2s\norm{L_\phi(\xi)})^k.
\end{equation*}
As $\binom{n+k}{n}\leq (n+k)^n$, this series converges for $\abs{s}<(2\norm{L_\phi(\xi)})^{-1}$. Hence $\eta$ is an analytic vector for $s_\T^0(\xi)$. Since $\bigcup_{n\in\mathbb N}\H_{\T,\leq n}$ is dense in $\F_\T(\H)$, it follows from Nelson's theorem that $s_\T^0(\xi)$ is essentially self-adjoint on $\bigcup_{n\in\mathbb N}\H_{\T,\leq n}$.
\end{proof}

\begin{definition}
    Let $(\H,\J,(\U_t))$ be a Tomita correspondence over $(M,\phi)$ and $\T$ a compatible twist on $\H\otimes_\phi \H$. The operator-valued twisted Araki--Woods algebra $\Gamma_\T(\H,\J,(\U_t))$ is the von Neumann algebra generated by the left action of $M$ and the operators $s_\T(\xi)$ for $\xi\in \dom(\U_{-i/2})$ left bounded with $\J\U_{-i/2}\xi=\xi$.
\end{definition}
If some of the operators $s_\T(\xi)$ are unbounded, this von Neumann algebra is to be understood as the von Neumann algebra generated by the left action of $M$ and resolvents of the operators $s_\T(\xi)$ for $\xi\in \H$ left-bounded with $\J\U_{-i/2}\xi=\xi$.

\begin{remark}
    If one views $\J$ and $\U_{-i}$ as analogs of the modular conjugation and modular operator in Tomita--Takesaki theory, then $\J\U_{-i/2}$ is the analog of the Tomita operator $S$. Since $\J\U_{-i/2}$ is an anti-linear involution, the space $\H_\IR$ of left-bounded vectors $\xi\in \H$ with $\J\U_{-i/2}\xi=\xi$ is a real-linear subspace of $\H$ such that $\H_{\IR}+i\H_{\IR}$ is dense in $\H$ and $\H_{\IR}\cap i\H_{\IR}=\{0\}$. So the choice of $\H_\IR$ can be seen as a bimodule analog of the choice of a real Hilbert subspace in the definition of scalar-valued twisted Araki--Woods and Gaussian algebras.
\end{remark}

\subsection{Modular Theory}

In this subsection we exhibit a natural weight on the operator-valued twisted Araki--Woods algebras that takes the role of the vacuum state in the scalar-valued theory (but it will only be a state if the reference state on the base algebra is a state). Then we show that the associated modular data can be described in terms of the underlying Tomita correspondence. For this to hold however, we need an additional condition on the twist, which we call locality. This condition ensures that the commutant of the operator-valued twisted Araki--Woods algebra on the twisted Fock bimodule also has a natural description in terms of the underlying Tomita correspondence.

Let $\T$ be a compatible braided twist on $\H\otimes_\phi\H$. Analogous to the construction of the left creation operators, we define for right bounded $\xi\in \H$ the right creation operator $b_{\T,n}^\ast(\xi)$ that acts on $\H_{\T,n}$ by $b_\T^\ast(\xi)\eta=\eta\otimes_\phi \xi$, and we write $b_{\T,n}(\xi)$ for the adjoint of $b_{\T,n-1}^\ast(\xi)$. Taking \Cref{lem:inn_prod_braided_twist} into account, the boundedness on $\H_{\T,n}$ follows from similar computations as for the left creation operators. An analogous argument to the one given in Lemma \ref{lem:field_op_ess_self-adj} shows that there is a unique self-adjoint operator $d_\T(\xi)$ on $\F_\T(\H)$ that coincides with $b^\ast_{\T,n}(\xi)+b_{\T,n}(\xi)$ on $\H_{\T,n}$ for all $n\in\mathbb N$.

\begin{lemma}\label{lem:intertwining_left_right_creation}
    Assume that $\T$ is a braided compatible twist. If $\xi\in \H$ is left bounded and $\J\U_{-i/2}\xi=\xi$, then $\J\xi$ is right bounded and $\J\U_{i/2}\J\xi=\J\xi$ and
    \begin{equation*}
        \J^{(n+1)}a^\ast_{\T,n}(\xi)\J^{(n)}=b^\ast_{\T,n}(\J\xi).
    \end{equation*}
\end{lemma}
\begin{proof}
    The part regarding left and right boundedness is \cite[Lemma 4.7]{Wir22b}. Moreover, $\J\U_{z}=\U_{\bar z}\J$ for $z\in\IC$ follows from the definition of Tomita correspondence. Thus $\J\U_{i/2}\J=\J(\J\U_{-i/2})$. This settles the first part.

    Moreover, if $\eta\in\H_{\T,n}$, then
    \begin{equation*}
        \J^{(n+1)}a^\ast_{\T,n}(\xi)\J^{(n)}\eta=\J^{(n+1)}(\xi\otimes_\phi\J^{(n)}\eta)=\eta\otimes_\phi\J\xi=b^\ast_{\T,n}(\J\xi)\eta.\qedhere
    \end{equation*}
\end{proof}

For the next definition recall that two (possibly unbounded) self-adjoint operators $x$ and $y$ are said to \emph{commute strongly} if $f(x)$ and $g(y)$ commute for all bounded Borel functions $f,g\colon \IR\to \IR$. If $x$ and $y$ are bounded, this notion reduces to $xy=yx$.

\begin{definition}
    We call a compatible braided twist on $\H\otimes_\phi\H$ \emph{local} if for all $\xi,\eta\in \H$ such that $\xi$ is left bounded and $\J\U_{-i/2}\xi=\xi$ and $\eta$ is right bounded and $\J\U_{i/2}\eta=\eta$, the operators $s_\T(\xi)$ and $d_\T(\eta)$ commute strongly.
\end{definition}

\begin{remark}
    As shown in the proof of \Cref{lem:field_op_ess_self-adj}, for every $n\in\mathbb N$ the subspace $L^2(M)\oplus \bigoplus_{j=1}^n \H_{\T,j}$ of $\F_\T(\H)$ consists of analytic vectors for $s_\T(\xi)$ for every left-bounded vector $\xi\in \dom(\U_{-i/2})$ with $\J\U_{-i/2}\xi=\xi$. The same is true if one replaces $s_\T(\xi)$ by $d_\T(\eta)$ for $\eta\in \dom(\U_{i/2})$ right-bounded with $\J\U_{i/2}\eta=\eta$. Therefore, the operators $s_\T(\xi)$ and $d_\T(\eta)$ commute strongly if and only if $[a_{\T,n}^\ast(\xi)+a_{\T,n}(\xi),b^\ast_{\T,n}(\eta)+b_{\T,n}(\eta)]=0$ for all $n\in\mathbb N$ by \cite[Corollary 9.2]{Nel59}.
\end{remark}

In the following we show that the twists discussed in the examples in the previous section are all local.

\begin{example}
The twist $\T=0$ is always local, as was proved in \cite[Lemma 4.10]{Wir22b}. Hence the free case of Shlyakhtenko's von Neumann algebras generated by operator-valued semicircular variables is covered by our setting (see also Subsection \ref{sec:Shlyakhtenko}).
\end{example}

\begin{example}
    As in Example \ref{Ex:twist_trivial_bimod} let $\H=L^2(M)\otimes \ell^2(I)$ with left and right action $x(\xi\otimes e_i)y=x\xi y\otimes e_i$, $\T:=1_{L^2(M)}\otimes T$ for a braided compatible twist $T$ on $\ell^2(I)\otimes\ell^2(I)$, $\J=J_\phi\otimes J$, where $J$ is the complex conjugation on $\ell^2(I)$, and $\U_t=\Delta_\phi^{it}\otimes U_t$ for some strongly unitary group $(U_t)$ on $\ell^2(I)$ such that $[U_t,J]=0$ for all $t\in\IR$. As discussed there, the operator $\T$ is a braided compatible twist under these assumptions.
    
    We next verify that $\T$ is local if and only if $T$ is local. Many interesting examples of local twists $T$ on $\ell^2(I)\otimes\ell^2(I)$ can be found in \cite{CL23}.

    A direct computation shows that $\J\U_{-i/2}$ is the closure of the operator with domain $\operatorname{lin}\{\xi\otimes \eta\mid \xi\in \dom(\Delta_\phi^{1/2}),\,\eta\in\dom(U_{-i/2})\}$ that acts as
    \begin{equation*}
        \xi\otimes\eta\mapsto J_\phi \Delta_\phi^{1/2}\xi\otimes JU_{-i/2}\eta.
    \end{equation*}
    Let $(e_i)$ be an orthonormal basis of the real Hilbert space $\ker(JU_{-i/2}-1)$. It follows that every $\xi\in \dom(\U_{-i/2})$ with $\J\U_{-i/2}\xi=\xi$ can be expressed as
    \begin{equation*}
        \xi=\sum_i\xi_i\otimes e_i
    \end{equation*}
    for vectors $\xi_i\in \dom(\Delta_\phi^{1/2})$ with $J_\phi\Delta_\phi^{1/2}\xi_i=\xi_i$ and $\sum_i \lVert \xi_i\rVert^2<\infty$. Moreover, $\sum_i \xi_i\otimes e_i$ is left-bounded if and only if $\xi_i$ is left-bounded for every $i$ and $\sum_i L_\phi(\xi_i)^\ast L_\phi(\xi_i)$ converges strongly. Hence every left-bounded vector $\xi\in\dom(\U_{-i/2})$ with $\J\U_{-i/2}\xi=\xi$ is of the form $\xi=\sum_i \Lambda_\phi(x_i)\otimes e_i$ with $x_i\in \n_\phi$ self-adjoint such that $\sum_i x_i^2$ converges strongly.

    An analogous argument shows that there exists an orthonormal basis $(f_i)$ of the real Hilbert space $\ker(J U_{i/2}-1)$ with $JU_{i/2}f_i=f_i$ such that every right-bounded vector $\eta\in \dom(\U_{i/2})$ with $\J\U_{i/2}\eta=\eta$ is of the form $\eta=\sum_i \Lambda_\phi^\prime(y_i)\otimes f_i$ with $y_i\in\n_\phi^\ast$ self-adjoint such that $\sum_i y_i^2$ converges strongly.

    For $\xi$ and $\eta$ as above we have
    \begin{align*}
        a_{\T,n}^\ast(\xi)+a_{\T,n}(\xi)&=\sum_i x_i\otimes s_{T,n}(e_i),\\
        b_{\T,n}^\ast(\xi)+b_{\T,n}(\xi)&=\sum_i J_\phi y_i J_\phi\otimes d_{T,n}(f_i).
    \end{align*}
    Since $J_\phi M J_\phi=M^\prime$, the two operators commute if and only if $[s_{T,n}(e_i),d_{T,n}(f_j)]=0$ for all $i,j$. Hence  $T$ is local if and only if $\T$ is local.
    \end{example}
\begin{example}
   Note in Example \ref{Ex:twist_group_alg} that $(\H_\pi,\J_\pi,(\U_t))$ is a Tomita correspondence over $(L(G),\tau)$ and $\T:=1_{\ell^2(G)}\otimes T$ on $\H_\pi\otimes \H_\pi$ is compatible if and only if $T$ is compatible. Further, as in the example above, one can verify that $\T$ is local if and only $T$ is local. 
\end{example}
\begin{example}
Let $(\H,\J,(\U_t))$ be a Tomita correspondence over $(M,\phi)$ such that $\H$ is centered and $\T$ be the flip map on $\H\otimes_\phi\H$. Then $q\T$ for $q\in[-1,1]$ is compatible as discussed in \Cref{Ex:Twist_Centered_Bimod}. Further, it can be shown as in the case of $q$-Gaussian algebras \cite{BKS97} that $q\T$ is local.  
\end{example}
We now turn to the weight on $\Gamma_\T(\H,\J,(\U_t))$ induced by $\phi$ and its modular theory. We first note that $M\subset \Gamma_\T(\H,\J,(\U_t))$ is an inclusion with expectation.

Let $\iota_\T\colon L^2(M)\to \F_\T(\H)$ denote the inclusion map and 
\begin{equation*}
    E_\T\colon \Gamma_\T(\H,\J,(\U_t))\to B(L^2(M)),\,x\mapsto\iota_\T^\ast x\iota_\T
\end{equation*}
Since $\iota_\T$ is a bimodule map, $E_\T$ maps bounded right module maps on $\F_\T(\H)$ to bounded right module maps on $L^2(M)$. In particular, $E_\T(\Gamma_\T(\H,\J,(\U_t)))= M$ and $E_\T$ is a normal conditional expectation. Let $\hat\phi_\T=\phi\circ E_\T$ on $\Gamma_\T(\H,\J,(\U_t))$.

To show that $E_\T$ and $\hat\phi_\T$ are faithful, we will use the following two lemmas.
\begin{lemma}[{\cite[Lemma IX.3.3]{Tak03}}]
    Let $\mathcal{K}$ be a correspondence from M to itself. The map $L_\phi$ is a bijection from the set of left $\phi$-bounded vectors in $\mathcal{K}$ onto the set of all right module maps $T\colon L^2(M)\to \mathcal{K}$ that satisfy $\phi(T^\ast T)<\infty$.
\end{lemma}

\begin{lemma}[{\cite[Lemma 4.14]{Wir22b}}]\label{lem:norm_left_bdd_vector}
     Let $\mathcal{K}$ be a correspondence from M to itself. If $\xi\in \mathcal{K}$ is left $\phi$-bounded, then $\phi(L_\phi(\xi)^\ast L_\phi(\xi))=\norm{\xi}_\mathcal{K}^2.$
\end{lemma}

\begin{proposition}\label{prop:standard_form_Araki-Woods}
    If $\T$ is a local braided compatible twist on $\H\otimes_\phi \H$, then the map
    \begin{equation*}
        \Lambda_\T\colon n_{\hat \phi_\T}\to \F_\T(\H),\,x\mapsto L_\phi^{-1}(x\iota_\T)
    \end{equation*}
    is injective and has dense range. Moreover, if $x,y\in \n_{\hat\phi_\T}$, then
    \begin{equation*}
        \langle\Lambda_\T(x),\Lambda_\T(y)\rangle=\hat\phi_\T(x^\ast y)
    \end{equation*}
    and
    \begin{equation*}
        \Lambda_\T(xy)=x\Lambda_\T(y).
    \end{equation*}
    In particular, $E_\T$ and $\hat\phi_\T$ are faithful.
\end{proposition}
\begin{proof}
    First note that $x\iota$ is a right module map from $L^2(M)$ to $\F_\T(\H)$ and by definition of $\n_{\hat\phi_\T}$ we have
    \begin{equation*}
        \phi((x\iota_\T)^\ast(x\iota_\T))=\hat \phi_\T(x^\ast x)<\infty.
    \end{equation*}
    Thus $\Lambda_\T$ is well-defined. Moreover,
    \begin{equation*}
        \norm{\Lambda_\T(x)}^2=\phi((x\iota_\T)^\ast(x\iota_\T))=\hat\phi_{\T}(x^\ast x).
    \end{equation*}
    By polarization, we obtain $\langle\Lambda_\T(x),\Lambda_\T(y)\rangle=\hat\phi_\T(x^\ast y)$ for $x,y\in\n_{\hat\phi_\T}$.

    Moreover, since $x$ is a right module map, 
    \begin{equation*}
        (xL_\phi^{-1}(y\iota_\T))m=x(L_\phi^{-1}(y\iota_\T)m)=xy\Lambda_\phi^\prime(m)
    \end{equation*}
    for all $m\in \n_\phi^\ast$. Thus $x L_\phi^{-1}(y\iota_\T)$ is left-bounded with $L_\phi(x L_\phi^{-1}(y\iota_\T))=xy\iota_\T$. Hence
    \begin{equation*}
        x\Lambda_\T(y)=xL_\phi^{-1}(y\iota_\T)=L_\phi^{-1}(xy\iota_\T)=\Lambda_\T(xy).
    \end{equation*} 
    Injectivity and density of the range can be shown as in \cite[Theorem 4.15]{Wir22b}.

    If $\hat\phi_\T(x^\ast x)=0$, the $\Lambda_\T(x)=0$ as shown above. Since $\Lambda_\T$ is injective, it follows that $x=0$. Hence $\hat\phi_\T$ and thus also $E_\T$ are faithful.
\end{proof}

As a consequence, the map $\Lambda_\T(x)\mapsto \Lambda_{\hat\phi_\T}(x)$ extends to a unitary operator from $\F_\T(\H)$ to $L^2((\Gamma_\T(\H,\J,(\U_t)),\hat\phi_\T)$. Thus $\hat\phi_\T$ is a normal semi-finite faithful weight on $\Gamma_\T(\H,\J,(\U_t))$ and the associated left Hilbert algebra is isomorphic to $\Lambda_\T(\n_{\hat\phi_\T}\cap \n_{\hat\phi_\T}^\ast)$ with product and involution given by $\Lambda_\T(x)\Lambda_\T(y)=\Lambda_\T(xy)$ and $\Lambda_\T(x)^\sharp=\Lambda_\T(x^\ast)$.

We will next characterize the associated modular theory. To do so, it is convenient to introduce the dual objects corresponding to the right field operators. We continue to assume that $\T$ is a local compatible braided twist.

Let $\Gamma_\T^\prime(\H,\J,(\U_t))$ be the von Neumann algebra generated by the right action of $M$ and the operators $d_\T(\eta)$ for $\eta\in \H$ right-bounded with $\J\U_{i/2}\eta=\eta$. The locality assumption implies $\Gamma_\T^\prime(\H,\J,(\U_t))\subset \Gamma_\T(\H,\J,(\U_t))^\prime$. Moreover, $E_\T(\Gamma^\prime_\T(\H,\J,(\U_t))=M^\prime$ and $\hat\psi_\T=\phi^\prime\circ E_\T$ defines a normal semifinite weight on $\Gamma_\T^\prime(\H,\J,(\U_t))$. Exactly as above one shows that $\hat\psi_\T$ is faithful and the associated semi-cyclic representation is given by $\Lambda_\T^\prime\colon \n_{\hat\psi_\T}\to\F_\T(\H),\,y\mapsto R_\phi^{-1}(y\iota_\T)$.

\begin{remark}
If $\phi$ is a state and $\Omega\in L^2_+(M)\subset \F_\T(\H)$ the cyclic vector representing $\phi$, then $\Lambda_\T(x)=x\Omega$ for $x\in \Gamma_\T(\H,\J,(\U_t))$ and likewise $\Lambda_\T^\prime(y)=y\Omega$ for $y\in \Gamma_\T^\prime(\H,\J,(\U_t))$. Furthermore, $\hat\phi_\T$ and $\hat\psi_\T$ are the restriction of the vector state $\langle\Omega,\cdot\,\Omega\rangle$ to $\Gamma_\T(\H,\J,(\U_t))$ and $\Gamma_\T^\prime(\H,\J,(\U_t))$, respectively. The conditional expectation $E_\T$ restricts to $x\mapsto L_\phi(\Omega)^\ast x L_\phi(\Omega)$ on $\Gamma_\T(\H,\J,(\U_t))$ and to $y\mapsto R_\phi(\Omega)^\ast y R_\phi(\Omega)$ on $\Gamma_\T^\prime(\H,\J,(\U_t))$.
\end{remark}

\begin{lemma}\label{lem:dual_Hilbert_alg}
Assume that $\T$ is a local braided compatible twist on $\H\otimes_\phi\H$. If $x\in \n_{\hat\phi_\T}$ and $y\in \n_{\hat\psi_\T}$, then
\begin{equation*}
    x\Lambda_\T^\prime(y)=y\Lambda_\T(x).
\end{equation*}
In particular, $\Lambda_\T^\prime(y)$ is a right-bounded vector for the left Hilbert algebra $\Lambda_\T(\n_{\hat\phi_\T}\cap\n_{\hat\phi_\T}^\ast)$ and $\pi_r(\Lambda_\T^\prime(y))=y$. Moreover, $\Lambda_\T^\prime(\n_{\hat\psi_\T}\cap\n_{\hat\psi_\T}^\ast)\subset \dom(S^\ast)$ and $S^\ast\Lambda_\T^\prime(y)=\Lambda_\T^\prime(y^\ast)$ for $y\in \n_{\hat\psi_\T}\cap\n_{\hat\psi_\T}^\ast$.
\end{lemma}
\begin{proof}
Since $\phi$ is semi-finite, there exists a net $(m_\alpha^\prime)$ of self-adjoint contractions in $\n_{\phi}\cap\n_{\phi}^\ast$ such that $m_\alpha^\prime\to 1$ strongly. Let
\begin{equation*}
    m_\alpha=\frac 1{\sqrt \pi}\int_{\IR}e^{-t^2}\sigma^\phi_t(m_\alpha^\prime)\,dt.
\end{equation*}
Then $m_\alpha\in\n_\phi\cap \n_\phi^\ast\cap \dom(\sigma^\phi_z)$ for all $z\in \IC$ and
\begin{equation*}
    \sigma^\phi_{z}(m_\alpha)=\frac 1 {\sqrt \pi}\int_{\IR}e^{-(t-z)^2}\sigma^\phi_t(m_\alpha^\prime)\,dt
\end{equation*}
for $z\in\IR$. It follows from the dominated convergence theorem that $\sigma^\phi_{z}(m_\alpha)\to 1$ strongly for all $z\in \IC$.

Now let $x\in\n_{\hat\phi_\T}$, $y\in \n_{\hat\psi_\T}$. If $n\in \n_\phi^\ast$, then
\begin{equation*}
    L_\phi^{-1}(x\iota_\T)m_\alpha n=L_\phi(L_\phi^{-1}(x\iota_\T))\Lambda_\phi^\prime(m_\alpha n)=x\sigma^\phi_{-i/2}(m_\alpha)\Lambda_\phi^\prime(n).
\end{equation*}
Thus $L_\phi^{-1}(x\iota_\T)m_\alpha=L_\phi^{-1}(x\iota_\T\sigma^\phi_{-i/2}(m_\alpha))=x\Lambda_\phi^\prime(m_\alpha)$. An analogous argument shows that $\sigma^\phi_{-i/2}(m_\alpha)R_\phi^{-1}(y\iota_\T)=y\Lambda_\phi(\sigma^\phi_{-i/2}(m_\alpha))$.

As $\T$ is assumed to be local, the operators $x$ and $y$ commute. Therefore,
\begin{align*}
    y\Lambda_\T(x)&=\lim_\alpha yL_\phi^{-1}(x\iota_\T)m_\alpha\\
    &=\lim_\alpha yx \Lambda_\phi^\prime(m_\alpha)\\
    &=\lim_\alpha xy\Lambda_\phi(\sigma_{-i/2}^\phi(m_\alpha))\\
    &=\lim_\alpha x\sigma^\phi_{-i/2}(m_\alpha) R_\phi^{-1}(y\iota_\T)\\
    &=x\Lambda_\T^\prime(y).
\end{align*}
The remaining statements follow from standard theory of left Hilbert algebras.
\end{proof}

We are now in the position to characterize the modular theory associated with the weight $\hat\phi_\T$. For the proof recall that $\H_0=\{\xi\in\bigcap_{z\in\IC}\dom(\U_z):\U_z\xi\text{ bounded for all }z\in\IC\}$, which is dense in $\H$ by \Cref{prop:bdd_analytic_vectors_dense}.

\begin{theorem}\label{thm:mod_theory_twisted}
    Let $\T$ be a local braided compatible twist on $\H\otimes_\phi\H$. The modular operator $\Delta_{\T}$ and the modular conjugation $J_\T$ associated with the left Hilbert algebra $\Lambda_\T(\n_{\hat\phi_\T}\cap\n_{\hat\phi_\T}^\ast)$ satisfy
    \begin{itemize}
        \item[(a)]$\Delta_\T^{it}=\F_\T(\U_t)$ for all $t\in\IR$,
        \item[(b)]$J_\T=\F_\T(\J)$.
    \end{itemize}
\end{theorem}
\begin{proof}
    Since $\hat\phi_\T=\phi\circ E_\T$ and $E_\T$ is a faithful normal conditional expectation, we have $J_\T\xi=J_\phi \xi$ and $\Delta_\T^{it}\xi=\Delta_\phi^{it}\xi$ for $\xi\in L^2(M)$.

    (a) We write $S_\T=J_\T\Delta_\T^{1/2}$. Let $\xi\in \H_0$ with $\J\U_{-i/2}\xi=\xi$ and let $p_R(\xi)=\1_{[-R,R]}(s_\T(\xi))$. Note that $p_R(\xi)s_\T(\xi)\iota_\T=p_R(\xi)L_\phi(\xi)$. Thus
    \begin{equation*}
        \hat\phi_\T((p_R(\xi)s_\T(\xi))^2)=\phi(\abs{p_R(\xi)s_\T(\xi)\iota_\T}^2)\leq \phi(L_\phi(\xi)^\ast L_\phi(\xi))=\norm{\xi}^2
    \end{equation*}
    by \Cref{lem:norm_left_bdd_vector}.

    As $p_R(\xi)s_\T(\xi)$ is self-adjoint, we have $p_R(\xi)s_\T(\xi)\in \n_{\hat\phi_\T}\cap\n_{\hat\phi_\T}^\ast$. Moreover,
    \begin{equation*}
        \Lambda_\T(p_R(\xi)s_\T(\xi))=L_\phi^{-1}(p_R(\xi)s_\T(\xi)\iota_\T)=p_R(\xi)\xi.
    \end{equation*}
    Hence $p_R(\xi)\xi\in \dom(S_\T)$ and $S_\T(p_R(\xi)\xi)=p_R(\xi)\xi$. Since $S_\T$ is closed, it follows that $\xi\in \dom(S_\T)$ and $S_\T\xi=\xi$. For arbitrary $\xi\in \H_0$ let $\eta=\frac 1 2(\xi+\J\U_{-i/2}\xi)$ and $\zeta=\frac 1{2i}(\xi-\J\U_{-i/2}\xi)$ so that $\eta,\zeta\in\H_0$, $\J\U_{-i/2}\eta=\eta$, $\J\U_{-i/2}\zeta=\zeta$ and $\xi=\eta+i\zeta$. By linearity, $\xi\in \dom(S_\T)$ and 
    \begin{equation*}
    S_\T\xi=S_\T\eta-iS_\T\zeta=\eta-i\zeta=\J\U_{-i/2}\xi.
    \end{equation*}
    Now let $\xi\in \H_0$ with $\J\U_{i/2}\xi=\xi$ and let $q_R(\xi)=\1_{[-R,R]}(d_\T(\xi))d_\T(\xi)$. Arguing as above, we get
    \begin{equation*}
        \hat\psi_\T((q_R(\xi)d_\T(\xi))^2)=\hat\phi^\prime(\abs{q_R(\xi)d_\T(\xi)\iota_\T}^2)\leq\phi^\prime(R_\phi(\xi)^\ast R_\phi(\xi))=\phi(J_\phi R_\phi(\xi)^\ast R_\phi(\xi) J_\phi).
    \end{equation*}
    
    By \cite[Lemma 4.7]{Wir22b}, we have $\J R_\phi(\xi) J_\phi=L_\phi(\J\xi)$. Thus
    \begin{equation*}
        \hat\psi_\T((q_R(\xi)d_\T(\xi))^2)=\phi(L_\phi(\J\xi)^\ast L_\phi(\J\xi))=\norm{\xi}^2.
    \end{equation*}
    It follows that $q_R(\xi)d_\T(\xi)\in \n_{\hat\psi_\T}\cap \n_{\hat\psi_\T^\prime}$. By \Cref{lem:dual_Hilbert_alg}, we have $q_R(\xi)\xi=\Lambda_\T^\prime(q_R(\xi)d_\T(\xi))\in \dom(S_\T^\ast)$ and
    \begin{equation*}
        S_\T^\ast q_R(\xi)\xi=q_R(\xi)\xi,
    \end{equation*}
    and we conclude as above that $\H_0\subset \dom(S_\T^\ast)$ and $S_\T^\ast \xi=\J\U_{i/2}\xi$ for $\xi\in \H_0$.

    If we combine these two facts, we obtain $\H_0\subset \dom(S_\T^\ast S_\T)$ and $S_\T^\ast S_\T\xi=\U_{-i}\xi$ for $\xi\in \H_0$. It follows by induction that $\H_0\subset \dom(\Delta_\T^\alpha)$ and $\Delta_\T^\alpha\xi=\U_{-i\alpha}\xi$ for all $\alpha\in\mathbb Q$. Thus $z\mapsto \U_z \xi$ and $z\mapsto \Delta_\T^{iz}\xi$ are analytic functions that coincide on $i\mathbb Q$. By the identity theorem for analytic function we obtain $\U_z\xi=\Delta_\T^{iz}\xi$ for all $z\in \IC$. As $\H_0\subset \H$ is dense by \Cref{prop:bdd_analytic_vectors_dense}, we conclude $\Delta^{it}_\T\xi=\U_t\xi$ for all $\xi\in \H$.
    
    In particular, if $\xi\in \H_0$ with $\J\U_{-i/2}\xi=\xi$, then
    \begin{align*}
    \Lambda_\T(\sigma^{\hat\phi_\T}_t(p_R(\xi)s_\T(\xi)))&=\Delta_\T^{it}\Lambda_\T(p_R(\xi)s_\T(\xi))\\
    &=\Delta_\T^{it}(p_R(\xi)\xi)\\
    &=\sigma^{\hat\phi_\T}_t(p_R(\xi))\Delta_{\T}^{it}\xi\\
    &=\sigma^{\hat\phi_\T}_t(p_R(\xi))\U_t\xi\\
    &=\Lambda_\T(\sigma^{\hat\phi_\T}_t(p_R(\xi))s_\T(\U_t\xi)),
    \end{align*}
    which implies $\sigma^{\hat\phi_\T}_t(p_R(\xi)s_\T(\xi))=\sigma^{\hat\phi_\T}_t(p_R(\xi))s_\T(\U_t\xi)$.
   
    We prove by induction over $n$ that $\Delta_\T^{it}\xi=\F_\T(\U_t)\xi$ for all $n\in\mathbb N$ and $\xi\in\H_{\T,n}$. The base case $n=1$ was settled in the previous paragraph. Assume that $\Delta_\T^{it}\xi=\F_\T(\U_t)\xi$ for all $\xi\in \H_{\T,k}$, $k\leq n$. If $\xi_1,\dots,\xi_{n+1}\in \H_0$ and $\J\U_{-i/2}\xi_1=\xi_1$, then
    \begin{align*}
        \Delta_\T^{it}(\xi_1\otimes_\phi\dots\otimes_\phi\xi_{n+1})&=\Delta_\T^{it}(s_\T(\xi_1)\xi_2\otimes_\phi\dots\otimes_\phi\xi_{n+1}-a_{\T,n}(\xi_1)\xi_2\otimes_\phi\dots\otimes_\phi\xi_{n+1}).
    \end{align*}
    Since $a_{\T,n}(\xi_1)\xi_2\otimes_\phi\dots\otimes_\phi\xi_{n+1}\in \H_{\T,n-1}$, we have by induction hypothesis
    \begin{equation*}
        \Delta_\T^{it}a_{\T,n}(\xi_1)\xi_2\otimes_\phi\dots\otimes_\phi\xi_{n+1}=\F_\T(\U_t)a_{\T,n}(\xi_1)\xi_2\otimes_\phi\dots\otimes_\phi\xi_{n+1}.
    \end{equation*}
    Moreover, if we use the induction hypothesis on the first summand, we obtain
    \begin{align*}
        \Delta_\T^{it}s_\T(\xi_1)\xi_2\otimes_\phi\dots\otimes_\phi\xi_{n+1}
        &=\lim_{R\to\infty}\Delta_\T^{it}(p_R(\xi_1)s_\T(\xi_1)\xi_2\otimes_\phi\dots\otimes_\phi\xi_{n+1})\\
        &=\lim_{R\to\infty}\sigma^{\hat\phi_\T}_t(p_R(\xi_1)s_\T(\xi_1))\F_\T(\U_t)\xi_2\otimes_\phi\dots\otimes_\phi\xi_{n+1}\\
        &=\lim_{R\to\infty}\sigma^{\hat\phi_\T}_t(p_R(\xi_1))s_\T(\U_t\xi_1)\F_\T(\U_t)\xi_2\otimes_\phi\dots\otimes_\phi\xi_{n+1}\\
        &=s_\T(\U_t\xi_1)\F_\T(\U_t)\xi_2\otimes_\phi\dots\otimes_\phi\xi_{n+1}\\
        &=\F_\T(\U_t)s_\T(\xi_1)\xi_2\otimes_\phi\dots\otimes_\phi\xi_{n+1},
    \end{align*}
    where we used in the last step that $\U_t^{(n+1)}a_{\T,n}^\ast(\xi_1)\U_{-t}^{(n)}=a_{\T,n}^\ast(\U_t\xi_1)$. If we reassemble the terms, we arrive at
    \begin{equation*}
        \Delta_\T^{it}(\xi_1\otimes_\phi\dots\otimes_\phi\xi_{n+1})=\F_\T(\U_t)(\xi_1\otimes_\phi\dots\otimes_\phi\xi_{n+1}).
    \end{equation*}
    As the set $\{\xi_1\otimes_\phi\dots\otimes_\phi\xi_{n+1}\mid \xi_1,\dots,\xi_{n+1}\in\H_0,\,\J\U_{-i/2}\xi_1=\xi_1\}$ is total in $\H_{\T,n}$, we conclude $\Delta_\T^{it}\xi=\F_\T(\U_t)\xi$ for all $\xi\in \H_{\T,n+1}$.

    (b) It was shown in (a) that $\H_0\subset \dom(S_\T)=\dom(\Delta_\T^{1/2})$ and $\S_\T\xi=\J\U_{-i/2}\xi$, $\Delta_\T^{1/2}\xi=\U_{-i/2}\xi$. Thus $J_\T\xi=\J\xi$ for $\xi\in \H_0$. The same identity for arbitrary $\xi\in \H$ follows by continuity of $\J$ and density of $\H_0$ in $\H$.

    If $\xi\in \H_0$ and $y\in \n_\phi$, then
    \begin{equation*}
        J_\T s_\T(\xi)J_\T\Lambda_\phi(y)=J_\T s_\T(\xi)\Lambda_\phi^\prime(y^\ast)=\J(\xi y^\ast)=y\J\xi=d_\T(\J\xi)\Lambda_\phi(y),
    \end{equation*}
    where we used that $J_\T$ coincides with $\J$ on $\H$. Let $x\in \Gamma_\T(\H,\J,(\U_t))$. Note that $J_\T s_\T(\xi) J_\T$ is affiliated with the commutant of $\Gamma_\T(\H,\J,(\U_t))$, and since $\T$ is local, the same is true for $d_\T(\J\xi)$. Thus
    \begin{align*}
        J_\T s_\T(\xi)J_\T\Lambda_\T(xy)&=J_\T s_\T(\xi) J_\T x\Lambda_\phi(y)\\
        &=x J_\T s_\T(\xi) J_\T \Lambda_\phi(y)\\
        &=xd_\T(\J\xi)\Lambda_\phi(y)\\
        &=d_\T(\J\xi)x\Lambda_\phi(y)\\
        &=d_\T(\J\xi)\Lambda_\T(xy).
    \end{align*}
    As $\Lambda_\T(\n_{\hat\phi_\T})$ is dense in $\F_\T(\H)$, we deduce $J_\T S_\T(\xi)J_\T=d_\T(\J\xi)$.

    Now we proceed again by induction. Assume that $J_\T\xi=\F_\T(\J)\xi$ for all $\xi\in \H_{\T,k}$, $k\leq n$, and let $\xi_1,\dots,\xi_{n+1}\in \H_0$ with $\J\U_{-i/2}\xi_1=\xi_1$. As above, we use that
    \begin{align*}
        \xi_1\otimes_\phi\dots\otimes_\phi\xi_{n+1}&=s_\T(\xi_1)\xi_2\otimes_\phi\dots\otimes_\phi\xi_{n+1}-a_{\T,n}(\xi_1)\xi_2\otimes_\phi\dots\otimes_\phi\xi_{n+1}.
    \end{align*}
    The second summand is in $\H_{\T,n-1}$, hence 
    \begin{equation*}
        J_\T(a_{\T,n}(\xi_1)\xi_2\otimes_\phi\dots\otimes_\phi\xi_{n+1})=\F_\T(\J)a_{\T,n}(\xi_1)\xi_2\otimes_\phi\dots\otimes_\phi\xi_{n+1}
    \end{equation*}
    by induction hypothesis. Moreover,
    \begin{align*}
        J_\T s_\T(\xi_1)\xi_2\otimes_\phi\dots\otimes_\phi\xi_{n+1}&=J_\T s_\T(\xi_1) \J_\T \J_\T\xi_2\otimes_\phi\dots\otimes_\phi\xi_{n+1}\\
        &=d_\T(\J\xi_1)\F_\T(\J)\xi_2\otimes_\phi\dots\otimes_\phi\xi_{n+1}.
    \end{align*}
    By \Cref{lem:intertwining_left_right_creation} we have $d_\T(\J\xi_1)\F_\T(\J)=\F_\T(\J)s_\T(\xi)$. Hence
    \begin{equation*}
        J_\T s_\T(\xi_1)\xi_2\otimes_\phi\dots\otimes_\phi\xi_{n+1}=\F_\T(\J)s_\T(\xi_1)\xi_2\otimes_\phi\dots\otimes_\phi\xi_{n+1}.
    \end{equation*}
    We have thus shown
    \begin{equation*}
        J_\T\xi_1\otimes_\phi\dots\otimes_\phi\xi_{n+1}=\F_\T(\J)\xi_1\otimes_\phi\dots\otimes_\phi\xi_{n+1}.
    \end{equation*}
    As in (a), a density argument concludes the induction step.
\end{proof}

\section{Examples of Operator-Valued Twisted Araki--Woods Algebras}

In this section we present some examples for which our operator-valued twisted Araki--Woods algebras can be identified with known constructions of von Neumann algebras. In particular, we discuss the connection to the von Neumann algebras generated by operator-valued semicircular variables discussed by Shlyakhtenko (Subsection \ref{sec:Shlyakhtenko}), give a full description of operator-valued twisted Araki--Woods algebras over type I factors as tensor products of the base algebra with scalar-valued twisted Araki--Woods factors (Subsection \ref{sec:Araki-Woods_type_I}) and show that semidirect products of $q$-Gaussian algebras (or more generally twisted Araki--Woods algebras) by Gaussian actions naturally occur as operator-valued twisted Araki--Woods algebras (Subsection \ref{sec:crossed_prod}).

All these identifications hold not only on the level of von Neumann algebras, but on the level of operator-valued $W^\ast$-probability spaces. An operator-valued $W^\ast$-probability space is a pair $(M\subset N,E)$ consisting of a unital inclusion of von Neumann algebras $M\subset N$ and a faithful normal conditional expectation $E\colon N\to M$. If $(M_i\subset N_i,E_i)$, $i\in \{1,2\}$, are operator-valued $W^\ast$-probability spaces, we say that $(M_1\subset N_1,E_1)$ and $(M_2\subset N_2,E_2)$ are isomorphic or write $(M_1\subset N_1,E_1)\cong (M_2\subset N_2,E_2)$ if there exists a $\ast$-isomorphism $\alpha\colon N_1\to N_2$ such that $\alpha\circ E_1=E_2\circ \alpha$. In particular, this implies $\alpha(M_1)=M_2$.

\subsection{Connection to Shlyakhtenko's von Neumann algebras Generated by Operator-Valued Semicircular Variables}\label{sec:Shlyakhtenko}

Let us briefly review Shlyakhtenko's construction from \cite{Shl99}. Let $M$ be a von Neumann algebra and $I$ a finite or countably infinite index set. A family $(\eta_{ij})_{i,j\in I}$ of normal linear maps on $M$ is called a \emph{covariance matrix} if the map
\begin{equation*}
    \eta\colon M\to M\overline\otimes B(\ell^2(I)),\,x\mapsto \sum_{i,j\in I}\eta_{i,j}(x)\otimes E_{ij}
\end{equation*}
is normal and completely positive. Here $E_{ij}$ denotes the canonical matrix units in $B(\ell^2(I))$ defined by $E_{ij}\delta_k=\delta_{jk}\delta_i$.

If $\eta$ is a covariance matrix, there exists a unique pair $(F,(\xi_i)_{i\in I})$ consisting of a von Neumann bimodule $F$ over $M$ and a family $(\xi_i)_{i\in I}$ in $F$ generating $F$ as von Neumann bimodule such that
\begin{equation*}
    \langle \xi_i\vert x\xi_j\rangle=\eta_{ij}(x)
\end{equation*}
for all $x\in M$, $i,j\in I$.


If $\phi$ is a faithful normal state on $M$, let $F_\phi$ be the completion of $F$ with respect to the inner product
\begin{equation*}
    \langle\zeta_1,\zeta_2\rangle_{F_\phi}=\phi(\langle\zeta_1\vert\zeta_2\rangle).
\end{equation*}
Note that $F_\phi$ is canonically isomorphic to $F\overline\odot_M L^2(M,\phi)$ via the map $\xi\mapsto \xi\otimes\Omega_\phi$, where $\Omega_\phi\in L^2_+(M,\phi)$ is the cyclic and separating vector associated with $\phi$. Through this isomorphism, $F_\phi$ becomes a correspondence from $M$ to itself. Note that the right module structure on $F_\phi$ is not the one inherited from $F$. In the following, expressions of the form $\xi x$ with $\xi\in F_\phi$ and $x\in M$ always refer to the right module structure of $F_\phi$ as correspondence, so there should be no confusion. 

Shlyakhtenko's operator-valued semicircular operators with variance $\eta$ are the operators $X_i^\phi=a^\ast_0(\xi_i)+a_0(\xi_i)$ on $\F_0(F_\phi)$. Let $\Phi(M,\eta)$ denote the von Neumann algebra generated by $M$ and $\{X^\phi_i\mid i\in I\}$.

There is a normal conditional expectation $E$ from $\Phi(M,\eta)$ onto $M$ given by $E(x)=\iota^\ast x \iota$, where $\iota\colon L^2(M)\to \F_0(\H)$ is the natural inclusion map. Without further assumptions, $E$ may fail to be faithful. A necessary and sufficient condition for the faithfulness of $E$ was given in \cite[Proposition 5.2]{Shl99}.

\begin{proposition}
If $E$ is faithful, then there exists an anti-unitary involution $\J\colon F_\phi\to F_\phi$ and a strongly continuous unitary group $(\U_t)$ on $F_\phi$ such that $(F_\phi,\J,(\U_t))$ is a Tomita correspondence over $(M,\phi)$ and $\Phi(M,\eta)=\Gamma_0(F_\phi,\J,(\U_t))$.\end{proposition}
\begin{proof}
    Let $J$, $\Delta$ denote the modular conjugation and modular operator associated with the faithful normal state $\phi\circ E$ on $\Phi(M,\eta)$. By \cite[Lemma 5.1]{Shl99}, both $J$ and $\Delta^{it}$ leave $F_\phi$ invariant and they restrict to $J_\phi$ and $\Delta_\phi^{it}$ on $L^2(M)$. Let $\J=J|_{F_\phi}$, $\U_t=\Delta^{it}\vert_{F_\phi}$. It is easy to see that $(F_\phi,\J,(\U_t))$ is a Tomita correspondence.

    Let $\Omega\in L^2_+(M)\subset \F_0(F_\phi)$ denote the cyclic vector such that $\phi=\langle\Omega,\cdot\,\Omega\rangle$. Note that 
    \begin{equation*}
        \J\U_{-i/2}\xi_i=J\Delta^{1/2}X^\phi_i\Omega=X^\phi_i\Omega=\xi_i.
    \end{equation*}
    Therefore $\Phi(M,\eta)\subset \Gamma_0(F_\phi,\J,(\U_t))$.

    For the converse inclusion we follow the argument from the proof of \cite[Lemma 5.1]{Shl99}. Let $K$ be the real-linear span of all elements of the form $x^\ast\xi_i \sigma^{\phi}_{i/2}(y)+y^\ast\xi_i \sigma^\phi_{i/2}(x)$ with $x,y\in \dom(\sigma^\phi_{i/2})$ and $i\in I$ and let $K'$ be the set of all left bounded vectors in $\dom(\U_{-i/2})$ with $\J\U_{-i/2}\xi=\xi$.

    By definition, $\xi_i\in K^\prime$. Since
    \begin{align*}
        \J\U_{-i/2}(x^\ast\xi_i \sigma^{\phi}_{i/2}(y)+y^\ast\xi_i \sigma^\phi_{i/2}(x))
        &=\J(\sigma^\phi_{i/2}(x)^\ast (\U_{-i/2}\xi_i)y+\sigma^\phi_{i/2}(y)^\ast(\U_{-i/2}\xi_i)x)\\
        &=y^\ast \xi_i\sigma^\phi_{i/2}(x)+x^\ast \xi_i \sigma^\phi_{i/2}(y),
    \end{align*}
    we have $K\subset K^\prime$.

  Note that
   \begin{equation*}
       (x^\ast X_i^\phi y+y^\ast X_i^\phi x)\Omega=x^\ast X_i^\phi\Omega \sigma^\phi_{i/2}(y)+y^\ast X_i^\phi \Omega \sigma^\phi_{i/2}(x)=x^\ast \xi_i \sigma^\phi_{i/2}(y)+y^\ast \xi_i\sigma^\phi_{i/2}(x)
   \end{equation*}
    and $s_0(\xi)\Omega=\xi$ for $\xi\in K^\prime$.  Since $\langle\Omega,\cdot\,\Omega\rangle$ is a faithful normal state on $\Gamma_0(F_\phi,\J,(\U_t))$ by \Cref{prop:standard_form_Araki-Woods}, we deduce $x^\ast X_i^\phi y+y^\ast X_i^\phi x=s_0(x^\ast \xi_i\sigma^\phi_{i/2}+y^\ast \xi_i\sigma^\phi_{i/2}(x))$.
    
    As $\{\zeta\in K:\norm{L_{\phi\circ E}(\zeta)}\leq 1\}$ spans $F_\phi$ as a complex Hilbert space, it follows that $\{\zeta\in K:\norm{L_{\phi\circ E}(\zeta)}\leq 1\}$ is dense in $\{\zeta\in K^\prime:\norm{s(\zeta)}\leq 1\}$. Hence the unit ball of $\Phi(M,\zeta)$ is strongly dense in $\Gamma_0(F_\phi,\J,(\U_t))$. As they are both von Neumann algebras, we conclude $\Phi(M,\eta)=\Gamma_0(F_\phi,\J,(\U_t))$.
\end{proof}

\begin{proposition}
If $M$ is a separable von Neumann algebra, $\phi$ a faithful normal state on $M$ and $(\H,\J,(\U_t))$ is a separable Tomita correspondence over $(M,\phi)$, then there exists a countable set $I$ and a covariance matrix $\eta\colon M\to M\overline\otimes B(\ell^2(I))$ such that $(M\subset\Gamma_0(\H,\J,(\U_t)),E_0)\cong (M\subset\Phi(M,\eta),E)$.
\end{proposition}
\begin{proof}
     Since $\H$ is separable, we can choose a countable dense subset $\{\zeta_i\mid i\in I\}$ of $\{\zeta\in \dom(\U_{-i/2})\mid \J\U_{-i/2}\zeta=\zeta,\norm{L_\phi(\zeta)}\leq 1\}$. Let
     \begin{equation*}
         \eta_{ij}\colon M\to M,\,\eta_{ij}(x)=L_\phi(\zeta_i)^\ast L_\phi(x\zeta_j)
     \end{equation*}
Clearly, $\eta$ is a covariance matrix. Form the von Neumann bimodule $F$ and the family $(\xi_i)_{i\in I}$ as in the definition of $\Phi(M,\eta)$. In particular, $\langle\xi_i\vert x\xi_j\rangle=L_\phi(\zeta_i)^\ast L_\phi(x\zeta_j)$. Since
         \begin{align*}
         \langle \xi_i,x\xi_j y\rangle_{F_\phi}&=\langle\xi_i\otimes\Omega_\phi,x\xi_j\otimes \Omega_\phi y\rangle_{F\overline{\otimes}_M L^2(M)}\\
         &=\langle \Omega_\phi,\langle \xi_i\vert x\xi_j\rangle \Omega_\phi y\rangle_{L^2(M)}\\
        &=\phi(\langle \xi_i\vert x\xi_j\rangle \sigma^\phi_{-i/2}(y))\\
       &=\phi(L_\phi(\zeta_i)^\ast L_\phi(x\zeta_j)\sigma^\phi_{-i/2}(y))\\
       &=\phi(L_\phi(\zeta_i)^\ast L_\phi(x\zeta_j y))\\
         &=\langle \zeta_i,x\zeta_j y\rangle_\H
     \end{align*}
     for $x\in M$, $y\in\dom(\sigma^\phi_{-i/2})$ and $i,j\in I$, there exists a unique unitary bimodule map $U\colon F_\phi\to \H$ such that $U\xi_i=\zeta_i$.  Then it is not hard to see that $\F(U)X^\phi_i\F(U)^\ast=s_0(\zeta_i)$.

     As $\F(U)$ is a unitary bimodule map that leaves $L^2(M)$ invariant, we conclude that $\F(U)\,\cdot\,\F(U)^\ast$ is a $\ast$-isomorphism between $\Phi(M,\eta)$ and $\Gamma_0(\H,\J,(\U_t))$ that leaves the copies of $M$ invariant and intertwines the canonical conditional expectations.
\end{proof}

\subsection{Operator-Valued Twisted Araki--Woods Algebras over Type I Factors}\label{sec:Araki-Woods_type_I}

In this section we give a complete description of operator-valued twisted Araki--Woods algebras in the case when the base algebra is a type I factor (and the Tomita correspondence is separable). More precisely, we show that they decompose as a tensor product of the base algebra and a (scalar-valued) twisted Araki--Woods factor. In particular, in the light of the previous subsection, this yields a complete description of the von Neumann algebras generated by operator-valued semicircular variables over type I factors for which the canonical conditional expectation is faithful.

\begin{theorem}\label{thm:twisted_Araki-Woods_type_I}
If $M$ is a type I factor, $\phi$ a normal semifinite faithful weight on $M$, $(H,\J,(\U_t))$ a separable Tomita correspondence over $(M,\phi)$ and $\T$ a local braided compatible twist on $\H\otimes_\phi \H$, then there exists a Hilbert space $H$, an anti-unitary involution $J$ on $H$, a strongly continuous unitary group $(U_t)$ on $H$ satisfying $[J,U_t]=0$ for all $t\in \IR$, and a local braided compatible twist $T$ on $H\otimes H$ such that
\begin{equation*}
    (M\subset\Gamma_\T(\H,\J,(\U_t)),E_0)\cong (M\subset M\overline\otimes \Gamma_T(H,J,(U_t)),\mathrm{id}\otimes \psi),
\end{equation*}
where $\psi$ is the vacuum state on the scalar-valued twisted Araki--Woods algebra $\Gamma_T(H,J,(U_t))$.
\end{theorem}
\begin{proof}
    By \Cref{thm:Tomita_corr_type_I} there exists a Hilbert space $H$, an anti-unitary involution $J\colon H\to H$, a strongly continuous unitary group $(U_t)$ on $H$ satisfying $[J,U_t]=0$ for all $t\in\IR$ and a unitary bimodule map $V\colon\H\to L^2(M)\otimes H$ such that $V\J=(J_\phi\otimes J)V$ and $V\U_t=(\Delta_\phi^{it}\otimes U_t)V$.

    As discussed in \Cref{ex:twist_trivial_bimod}, there is a canonical identification $(L^2(M)\otimes H)^{\otimes_\phi 2}\cong L^2(M)\otimes H\otimes H$, which we will tacitly make use in the following. The map $(V\otimes_\phi V)\T(V^\ast\otimes_\phi V^\ast)$ is a bounded bimodule map on $L^2(M)\otimes H\otimes H$, that is, it belongs to $(M\otimes \IC 1_{H\otimes H})^\prime\cap (M^\prime\otimes \IC 1_{H\otimes H})^\prime$. Since $M$ is a factor, this implies that there exists a bounded linear map $T$ on $H\otimes H$ such that $(V\otimes_\phi V)\T(V^\ast\otimes_\phi V^\ast)=1_{L^2(M)}\otimes T$.
    
    A direct computation shows that $V^{\otimes_k \phi}\T_k(V^\ast)^{\otimes_\phi k}=1_{L^2(M)}\otimes T_k$ and $V^{\otimes_\phi n}P_{\T,n}(V^\ast)^{\otimes_\phi n}=1_{L^2(M)}\otimes P_{T,n}$. Thus $T$ is a twist, and a strict twist if $\T$ is a strict twist. Moreover, $V^{\otimes_\phi n}$ extends to a unitary bimodule map $V^{(n)}$ between $H_{\T,n}$ and $L^2(M)\otimes H_{T,n}$ for all $n\in \mathbb N$. We write $\F_\T(V)$ for the unitary bimodule map between $\F_\T(\H)$ and $L^2(M)\otimes \F_T(H)$ that restricts to the identity on $L^2(M)$ and to $V^{(n)}$ on $\H_{\T,n}$.
    
    Furthemore,
    \begin{equation*}
        1_{L^2(M)}\otimes T_1 T_2 T_1=\F_\T(V)\T_1 \T_2 \T_1\F_\T(V)^\ast=\F_\T(V)\T_2 \T_1 \T_2\F_\T(V)^\ast=1_{L^2(M)}\otimes T_2 T_1 T_2.
    \end{equation*}
    Hence $T$ is braided. The compatibility of $T$ follows directly from the tensor product decompositions of $\T$, $\J$ and $(\U_t)$.

    To show that $T$ is local, let $\xi,\eta\in H$ and $x\in\n_{\phi}\cap \n_\phi^\ast$ and non-zero. A direct computation shows $a^\ast_{\T,n}(V^\ast(\Lambda_\phi(x)\otimes \xi))=\F_\T(V)(x\otimes a^\ast_{T,n}(\xi))\F_\T(V)^\ast$ and likewise $b^\ast_{\T,n}(V^\ast(\Lambda_\phi^\prime(x)\otimes\eta))=\F_\T(V)(J_\phi x^\ast J_\phi\otimes b^\ast_{T,n}(\eta))\F_\T(V)^\ast$. Hence the locality of $T$ follows from the locality of $\T$, and the spatial isomorphism induced by $\F_\T(V)$ maps $\Gamma_\T(\H,\J,(\U_t))$ onto $M\overline\otimes \Gamma_T(H,J,(U_t))$ and interwines the conditional expectations.
\end{proof}

\begin{corollary}\label{cor:op-valued_semicircular}
If $M$ is a type I factor, $I$ a countable set and $\eta\colon M\to M\overline\otimes B(\ell^2(I))$ a covariance matrix such that the canonical conditional expectation from $\Phi(M,\eta)$ onto $M$ is faithful, then $\Phi(M,\eta)$ is isomorphic to a tensor product of $M$ with a free Araki--Woods algebra.
\end{corollary}

\begin{remark}
    By \cite[Corollary 6.11]{Shl97}, if $(U_t)$ is non-trivial, then $\Gamma_0(H,J,(U_t))$ is a type III factor and in particular properly infinite. Thus $M\overline\otimes \Gamma_0(H,J,(U_t))\cong \Gamma_0(H,J,(U_t))$ for every type I factor $M$. If $(U_t)$ is trivial, then $\Gamma_0(H,J,(U_t))\cong L(\mathbb F_{\dim H})$ and $M\overline\otimes \Gamma_0(H,J,(U_t))$ is an interpolated free group factor if $M$ is finite.
\end{remark}

\subsection{Gaussian Actions on Twisted Araki--Woods Algebras}\label{sec:crossed_prod}

Let $H$ be a Hilbert space, $J$ an anti-unitary involution on $H$, $(U_t)$ a strongly continuous unitary group on $H$ such that $[U_t,J]=0$ for all $t\in\IR$ and $T$ a local compatible braided twist on $H\otimes H$. We consider group actions on the twisted Araki--Woods algebra $\Gamma_T(H,J,(U_t))$ that are induced by unitary actions on $H$ compatible with $J$ and $(U_t)$.

More precisely, let $G$ be a discrete group and $\pi$ a unitary representation of $G$ on $H$ such that $[\pi(g),J]=0$ for all $g\in G$. In other words, $\pi$ is the complexification of an orthogoncal action of $G$ on $H^J$. If $[\pi(g)\otimes \pi(g),T]=0$ for all $g\in G$ and $[\pi(g),U_t]=0$ for all $g\in G$, $t\in \IR$, then $\pi(g)^{\otimes n}$ is a unitary operator on $H^{\otimes n}$ that commutes with $P_{T,n}$. Hence it gives rise to a unitary operator $\F_T(\pi(g))$ on $\F_T(H)$. The Gaussian action $\Gamma_T(\pi)$ of $G$ on the twisted Araki--Woods algebra $\Gamma_T(H,J,(U_t))$ induced by $\pi$ is given by
\begin{equation*}
    \Gamma_T(\pi(g))\colon B(\F_T(H))\to B(\F_T(H)),\,\Gamma_T(g)x=\F_T(\pi(g))x\F_T(\pi(g))^\ast.
\end{equation*}

\begin{lemma}
If $g\in G$ and $x\in \Gamma_T(H,J,(U_t))$, then $\Gamma_T(\pi(g))x\in \Gamma_T(H,J,(U_t))$.
\end{lemma}
\begin{proof}
Since $T$ is local, it suffices to show that $\Gamma_T(\pi(g))(b_T^\ast(\eta)+b_T(\eta))\in \Gamma_T(H,J,(U_t))^\prime$ for all $\eta\in \dom(U_{i/2})$ with $JU_{i/2}\eta=\eta$. For $\xi_1,\dots,\xi_n\in H$ we have
\begin{align*}
    \Gamma_T(\pi(g))(b_T^\ast(\eta))(\xi_1\otimes\dots\xi_n)&=\F_T(\pi(g))(\pi(g)^\ast \xi_1\otimes \pi(g)^\ast\xi_n\otimes \eta)\\
    &=\xi_1\otimes\dots\otimes\xi_n\otimes \pi(g)\eta\\
    &=b_T^\ast(\pi(g)\eta)(\xi_1\otimes\dots\otimes\xi_n).
\end{align*}
Hence $\Gamma_T(\pi(g))(b_T^\ast(\eta)+b_T(\eta))=b_T^\ast(\pi(g)\eta)+b_T(\pi(g)\eta)$. Since $\pi(g)$ commutes with $J$ and $(U_t)$ by assumption, we have $\pi(g)\eta\in \dom(U_{i/2})$ and $JU_{i/2}\pi(g)\eta=\pi(g)\eta$. Thus $b_T^\ast(\pi(g)\eta)+b_T(\pi(g)\eta)\in \Gamma_T(H,J,(U_t))^\prime$.
\end{proof}

In the following we consider the correspondence $\H_\pi$ from $L(G)$ to itself that was defined in \Cref{ex:twist_group_alg}. As discussed there, one can identify $\H_\pi^{\otimes_\tau n}$ with $\ell^2(G)\otimes H^{\otimes n}$, and we will tacitly do so in the following.

To fix notations, let us also briefly recall the definition of crossed products. If $N$ is a von Neumann algebra, $\Gamma$ a discrete group and $\rho\colon \Gamma\to \mathrm{Aut}(N)$ a group homomorphism, let
\begin{equation*}
    \alpha\colon N\to B(\ell^2(\Gamma)\otimes L^2(N)),\,\alpha(x)(\delta_\gamma\otimes \eta)=\delta_\gamma\otimes\rho(\gamma^{-1})x\eta.
\end{equation*}
The crossed product $N\rtimes_\rho\Gamma$ is the von Neumann algebra generated by $L(\Gamma)\otimes\IC 1$ and $\alpha(N)$ inside $B(\ell^2(\Gamma)\otimes L^2(N))$.

If $\psi$ is a faithful normal state on $N$ that is invariant under $\rho$, $\Omega_\psi\in L^2_+(N)$ the corresponding cyclic vector and 
\begin{equation*}
\iota_\psi\colon \ell^2(\Gamma)\to \ell^2(\Gamma)\otimes L^2(N),\,\delta_\gamma\mapsto \delta_\gamma\otimes \Omega_\psi,
\end{equation*}
then $E_\psi=\iota_\psi^\ast\,\cdot\,\iota_\psi$ is a faithful normal conditional expectation from $N\rtimes_\rho\Gamma$ onto $L(\Gamma)$. 

\begin{proposition}
Let $T$ be an equivariant local compatible braided twist on $H\otimes H$. If $\T=1_{\ell^2(G))}\otimes T$, $\U_t=1_{\ell^2(G)}\otimes U_t$ and $\psi$ is the vacuum state on $\Gamma_T(H,J,(U_t))$, then the operator-valued $W^\ast$-probability spaces $(L(G)\subset\Gamma_{\T}(\H_\pi,\J_\pi,(\U_t)),E)$ and $(L(G)\subset\Gamma_T(H,J,(U_t))\rtimes_{\Gamma_T(\pi)}G,E_\psi)$ are isomorphic.
\end{proposition}
\begin{proof}
Let
\begin{equation*}
    V_n\colon \ell^2(G)\otimes H_{T,n}\to \ell^2(G)\otimes H_{T,n},\,\delta_g\otimes\xi\mapsto \delta_g\otimes\pi(g^{-1})^{\otimes n}\xi.
\end{equation*}
If $g,h\in G$ and $\xi,\eta \in H^{\otimes n}$, then
\begin{align*}
    \langle V_n(\delta_g\otimes\xi),V_n(\delta_h\otimes\eta)\rangle_{\ell^2(G)\otimes H_{T,n}}&=\delta_{g,h}\langle \pi(g^{-1})^{\otimes n}\xi,\pi(g^{-1})^{\otimes n}\eta\rangle_{H_{T,n}}\\
    &=\delta_{g,h}\langle\pi(g^{-1})^{\otimes n}\xi,\pi(g^{-1})^{\otimes n}P_{T,n}\eta\rangle_{H^{\otimes n}}\\
    &=\delta_{g,h}\langle \xi,P_{T,n}\eta\rangle_{H^{\otimes n}}\\
    &=\langle \delta_g\otimes\xi,\delta_h\otimes \eta\rangle_{\ell^2(G)\otimes H_{T,n}},
\end{align*}
where we used that $\pi(g^{-1})^{\otimes n}$ commutes with $P_{T,n}$ since $T$ is equivariant. As the range of $V_n$ is dense for all $n\in\mathbb N$, we conclude that $V_n$ is unitary for all $n\in\mathbb N$.

Moreover,
\begin{align*}
    V_n(\delta_{gh}\otimes \pi(g)^{\otimes n}\xi)=\delta_{gh}\otimes \pi(h^{-1})\xi=(\lambda_g\otimes 1)V_n(\delta_h\otimes \xi)
\end{align*}
and
\begin{align*}
    V_{n+1}((1\otimes a^\ast_T(\xi))(\delta_g\otimes \eta))&=V_{n+1}(\delta_g\otimes\xi\otimes\eta)\\
    &=\delta_g\otimes \pi(g^{-1})\xi\otimes\pi(g^{-1})^{\otimes n}\eta\\
    &=(1\otimes \Gamma_T(\pi(g^{-1}))a^\ast_T(\xi))(\delta_g\otimes\pi(g^{-1})^\otimes n\eta)\\
    &=(1\otimes \Gamma_T(\pi(g^{-1}))a^\ast_T(\xi))V_n(\delta_g\otimes\eta).
\end{align*}
Moreover, if $\xi\in \dom(U_{-i/2})$, then
\begin{equation*}
    \J_\pi\U_{-i/2}(\delta_e\otimes\xi)=\J_\pi(\delta_e\otimes U_{-i/2}\xi)=\delta_e\otimes JU_{-i/2}\xi.
\end{equation*}
Hence $\J_\pi\U_{-i/2}(\delta_e\otimes\xi)=(\delta_e\otimes\xi)$ if and only if $J U_{-i/2}\xi=\xi$.

Let
\begin{equation*}
    V\colon \ell^2(G)\otimes \F_T(H)\to \ell^2(G)\otimes\F_T(H),\,V=1_{\ell^2(G)}\oplus\bigoplus_{n=1}^\infty V_n.
\end{equation*}
Then $V$ is unitary and the previous computations show
\begin{align*}
    V(\lambda_g\cdot \xi)&=(\lambda_g\otimes 1)V\xi\\
    Vs_\T(\delta_e \otimes \xi)&=1\otimes \alpha(s_T(\xi))V,
\end{align*}
where $\alpha\colon \Gamma_T(H,J,(U_t))\to B(\ell^2(G)\otimes\F_T(H))$ is the representation in the definition of the crossed product. Thus $V\Gamma_\T(\H_\pi,\J_\pi,(\U_t))V^\ast=\Gamma_T(H,J,(U_t))\rtimes_{\Gamma_T(\pi)}G$ and $VL(G)V^\ast=L(G)\otimes\IC 1$. Furthermore, $V\circ\iota_\T=\iota_\psi$, which implies $E(\cdot)=E_\psi(V\,\cdot\,V^\ast)$.
\end{proof}

\begin{remark}
    In the free case $T=0$, the structure of the crossed product $\Gamma_0(H,J,(U_t))\rtimes_{\Gamma(\pi)}G$ has been studied in \cite{HT21}. In particular, the authors characterize factoriality, determine the type classification and give sufficient criteria for fullness and strong solidity for these crossed product algebras.
\end{remark}

\section{Factoriality}\label{sec:factoriality}

In this section we study when the operator-valued Araki--Woods algebras are factors. In contrast to the scalar-valued case, $\Gamma_\T(\H,\J,(\U_t))$ can fail to be a factor even for strict twists -- for example, if $\H=L^2(M)\otimes \IC^d$, $\J=J_\phi\otimes J$ and $\U_t=\Delta_\phi^{it}\otimes 1$, then it is not hard to see that $\Gamma_0(\H,\J,(\U_t))\cong M\overline\otimes L(\mathbb F_d)$ (compare with the proof of \Cref{thm:twisted_Araki-Woods_type_I}), which is only a factor if $M$ is a factor. However, there are also interesting cases when $\Gamma_\T(\H,\J,(\U_t))$ is a factor even though $M$ is not (see \cite[Section 3]{Shl99} or \Cref{ex:coarse_bimodule_factor} below). This already suggests that the factoriality question is more subtle than in the scalar-valued case.

Here we give two sufficient criteria for factoriality. The first is based on mixing properties of the Tomita correspondence, while the second one relies on a suitable notion of the Tomita correspondence having sufficiently large dimension as bimodule.

We start with the mixing property. Let $M$ be a von Neumann algebra. Following the terminology used for example in \cite{BBM19}, a correspondence $\H$ over $M$ is called \emph{jointly mixing} if $\pi_l^\H(x_n)\pi_r^\H(y_n)\to 0$ in the weak$^\ast$ topology whenever $(x_n)$, $(y_n)$ are bounded sequences in $M$ such that $x_n\to 0$ or $y_n\to 0$ in the weak$^\ast$ topology.

\begin{example}\label{ex:coarse_bimodule}
    If $\pi$ is a normal representation of $M$ on a Hilbert space $H$, then $\mathrm{HS}(H)$ with the left and right action given by $x\cdot R\cdot y=\pi(x)R\pi(y)$ is a jointly mixing correspondence from $M$ to itself. In particular, the coarse bimodule $L^2(M)\otimes L^2(M)$ is jointly mixing.
    
    Note that if $(\pi,H,(U_t))$ is a covariant representation of $(M,\sigma^\phi_t)$, that is, $U_t\pi(x)U_{-t}=\pi(\sigma^\phi_t(x))$, then one can turn $\mathrm{HS}(H)$ into a Tomita correspondence by defining $\J(R)=R^\ast$ and $\U_t(R)=U_t R U_{-t}$.
\end{example}

\begin{lemma}\label{lem:twisted_bimodule_mixing}
If $\H$ is a jointly mixing Tomita correspondence, then $\H_{\T,k}$ is jointly mixing for all $k\geq 1$ and all strict twists $\T$.
\end{lemma}
\begin{proof}
We first prove that the relative tensor product $\H\otimes_\phi\mathcal K$ is jointly mixing if $\mathcal K$ is jointly mixing. Let $(x_n)$, $(y_n)$ be bounded sequences in $M$ such that one of them converges to $0$ in the weak$^\ast$ topology.

Since the sequences are bounded, it suffices to show that
\begin{equation*}
\langle \xi_1\otimes_\phi\eta_1,x_n\xi_2\otimes_\phi\eta_2 y_n\rangle_{\H\otimes_\phi\mathcal K}\to 0
\end{equation*}
for all $\xi_1,\xi_2\in \H$ left $\phi$-bounded and $\eta_1,\eta_2\in \mathcal K$.

We have
\begin{equation*}
\langle \xi_1\otimes_\phi\eta_1,x_n\xi_2\otimes_\phi\eta_2 y_n\rangle_{\H\otimes_\phi\mathcal K}=\langle \eta_1,\pi_l^{\mathcal K}(L_\phi(\xi_1)^\ast x_n L_\phi(\xi_2))\pi_r^{\mathcal K}(y_n)\eta_2\rangle_{\mathcal K}.
\end{equation*}
By assumption, $L_\phi(\xi_1)^\ast x_n L_\phi(\xi_2)\to 0$ or $y_n\to 0$ in the weak$^\ast$ topology. As $\mathcal K$ is jointly mixing, it follows that $\langle \xi_1\otimes_\phi\eta_1,x_n\xi_2\otimes_\phi\eta_2 y_n\rangle_{\H\otimes_\phi\mathcal K}\to 0$.

It follows by induction that $\H^{\otimes_\phi k}$ is jointly mixing for all $k\geq 1$. To prove that $\H_{\T,k}$ is jointly mixing, let again $(x_n)$, $(y_n)$ be bounded sequences in $M$ such that one of them converges to $0$ in the weak$^\ast$ topology. Again, by a density argument it suffices to show that
\begin{equation*}
\langle \xi,x_n\eta y_n\rangle_{\T,k}\to 0
\end{equation*}
for all $\xi,\eta\in \H^{\otimes_\phi k}$. Since $\mathcal P_{\T,k}$ is a bimodule map, we have $\langle \xi,x_n\eta y_n\rangle_{\T,k}=\langle\xi,x_n(\mathcal P_{\T,k}\eta)y_n\rangle_{\H^{\otimes_\phi k}}$, and the claimed convergence follows immediately from the fact that $\H^{\otimes_\phi k}$ is jointly mixing.
\end{proof}

The following result is an adaptation of \cite[Theorem 5.1]{KV19}, where the result was proved under the assumption that $\phi$ is a trace, $\T=0$ and $\U_t=\mathrm{id}_\H$ (with a slightly different mixing condition). We call an element of $M$ \emph{diffuse} if it is self-adjoint and its spectral measure has no atoms.

\begin{proposition}\label{prop:mixing_factor}
Let $(\H,\mathcal J,(\mathcal U_t))$ be a jointly mixing Tomita correspondence over $(M,\phi)$ with $\phi$ finite and let $\T$ be a local compatible braided strict twist on $\H\otimes_\phi\H$. If the centralizer $M^\phi$ of $\phi$ contains a diffuse element, then 
\begin{equation*}
\mathcal Z(\Gamma_\T(\H,\mathcal J,(\mathcal U_t)))=\{z\in\mathcal Z(M)\mid z\xi=\xi z\text{ for all }\xi\in\H\}.
\end{equation*}
In particular, if $M$ is a factor, so is $\Gamma_\T(\H,\mathcal J,(\mathcal U_t))$.
\end{proposition}
\begin{proof}
If $x\in M^\phi$ is diffuse, then there exists a sequence $(u_n)$ of unitaries in $\{x\}^{\prime\prime}$ that converges to $0$ in the weak$^\ast$ topology. Let $z\in\mathcal Z(\Gamma_\T(\H,\mathcal J,(\mathcal U_t)))$ and let $\Omega\in L^2_+(M)$ denote the cyclic and separating vector representing $\phi$. We have
\begin{equation*}
(z-E_\T(z))\Omega=(u_n z u_n^\ast-E_\T(u_n z u_n^\ast))\Omega=u_n(z-E_\T(z))u_n^\ast\Omega=u_n(z-E_\T(z))\Omega u_n^\ast.
\end{equation*}
Note that $(z-E_\T(z))\Omega\in \F_\T(\H)\ominus L^2(M)$. It follows from \Cref{lem:twisted_bimodule_mixing} that $u_n(z-E(z))\Omega u_n^\ast\to 0$ weakly. Thus $z=E_\T(z)\in M\cap M^\prime$.

Moreover, $z\in \mathcal Z(\Gamma_M(\H,\mathcal J,(\mathcal U_t))$ implies
\begin{equation*}
z\xi=z s_\T(\xi)\Omega=s_\T(\xi)z\Omega=\xi z
\end{equation*}
for all bounded vectors $\xi\in \H$. As bounded vectors are dense in $\H$, it follows that $z\xi=\xi z$ for all $\xi\in \H$.

Conversely, if $x\in \mathcal Z(M)$ and $x\xi=\xi x$ for all $\xi\in \H$, then
\begin{equation*}
        x s_\T(\xi)\Omega=x\xi=\xi x=s_\T(\xi)J_\phi x^\ast \Omega=s_\T(\xi)x\Omega
    \end{equation*}
    for all left-bounded vectors $\xi\in \H$ with $\J\U_{-i/2}\xi=\xi$. Therefore, $x\in \mathcal Z(\Gamma_\T(\H,\J,(\U_t)))$.
\end{proof}

\begin{example}\label{ex:coarse_bimodule_factor}
    Let $(\pi,H,(U_t))$ be a normal faithful covariant representation of $(M,\sigma^\phi)$ and let $(\mathrm{HS}(H),\J,(\U_t))$ be the Tomita correspondence described in \Cref{ex:coarse_bimodule}. This correspondence is jointly mixing. Moreover, if $\pi(x)R=R\pi(x)$ for all $R\in\mathrm{HS}(H)$, then $\pi(x)\in \IC 1$, hence $x\in \mathbb C 1$ since $\pi$ was assumed to be faithful. Thus $\Gamma_\T(\mathrm{HS}(H),\J,(\U_t))$ is a factor if $M^\phi$ contains a diffuse element and $\T$ is a local compatible braided strict twist on $\H\otimes_\phi\H$.
\end{example}

For the second factoriality criterion we need the operator inequalities from \Cref{lem:bound_twisted_norm_tensor_product}. For the ease of notation, we write
\begin{equation*}
    c(q)=\prod_{k=1}^\infty\frac{1-q^k}{1+q^k},\quad d(q)=(1-q)^{-1}
\end{equation*}
for $0\leq q<1$. Note that these are the limits as $n\to \infty$ of the constants appearing in \Cref{lem:bound_twisted_norm_tensor_product}.

The following result is a direct consequence of \Cref{lem:bound_twisted_norm_tensor_product}.
\begin{lemma}\label{lem:bound_change_twisted_norm}
    Let $M$ be a von Neumann algebra, $\phi$ a normal semi-finite faithful weight on $M$ and $\H$ a correspondence from $M$ to itself. If $\T$ is a braided twist on $\H\otimes_\phi \H$ with $\norm{\T}<1$, then the identity map on $L^\infty(\H_M,\phi)\odot \F_\T(\H)$ (resp. $\F_\T(\H)\odot L^\infty(_M\H,\phi)$) extends to a bounded invertible operator $j_l$ (resp. $j_r$) from $\H\otimes_\phi\F_\T(\H)$ (resp. $\F_\T(\H)\otimes_\phi\H$) to $\F_\T(\H)\ominus L^2(M)$ with $\norm{j_{l/r}}\leq d(\norm{\T})^{1/2}$, $\norm{j_{l/r}^{-1}}\leq c(\norm{\T})^{-1/2}$.
\end{lemma}

\begin{proposition}[Spectral Gap]\label{prop:spectral_gap}
Let $(\H,\J,(\U_t))$ be a Tomita correspondence over $(M,\phi)$ and assume there exists a finite family $(\xi_i)_{i\in I}$ in $\H_0$ such that
\begin{itemize}
    \item $\J\xi_i=\xi_i$ for all $i\in I$,
    \item $L_\phi(\xi_i)^\ast L_\phi(\xi_j)=\delta_{ij}1$ for all $i,j\in I$,
    \item there exists a family $(p_i)_{i\in I}$ of projections in $M$ such that $L_\phi(\U_{i/2}\xi_i)^\ast L_\phi(\U_{i/2}\xi_j)=\delta_{ij}p_i$ for all $i,j\in I$.
\end{itemize}
There exists a function $f\colon [0,1)\to [0,\infty)$  such that whenever $\T$ is a local braided compatible twist on $\H\otimes_\phi\H$ with $\norm{\T}<1$ and $\abs{I}>f(\norm{\T})$, then there exists a constant $\kappa>0$ such that
\begin{equation*}
    \sum_{i\in I}\lVert \F_\T(\J)x^\ast \F_\T(\J)\xi_i-x\xi_i\rVert_{\F_\T(\H)}^2\geq \kappa^2 \hat\phi_\T(\abs{x-E_\T(x)}^2)
\end{equation*}
for all $x$ in the centralizer of $\hat\phi_\T$.
\end{proposition}
\begin{proof}
The idea of the proof follows \cite{Sni04}. First note that the assumptions imply $\J\U_{-i/2}\xi_i=\U_{i/2}\xi_i$, $\J\U_{i/2}\xi_i=\U_{-i/2}\xi_i$ and $R_\phi(\xi_i)^\ast R_\phi(\xi_j)=\delta_{ij}1$, $R_\phi(\U_{-i/2}\xi_i)^\ast R_\phi(\U_{-i/2}\xi_j)=\delta_{ij}J_\phi p_i J_\phi$.

Define 
\begin{align*}
S&\colon \mathcal F_\T(\H)\to \ell^2(I)\otimes\mathcal F_\T(\H),\,S \eta=\sum_{i\in I}e_i\otimes (a_\T(\U_{i/2}\xi_i)-b_\T(\U_{-i/2}\xi_i))\eta,\\
T&\colon \mathcal F_\T(\H)\to \ell^2(I)\otimes\mathcal F_\T(\H),\,T \eta=\sum_{i\in I}e_i\otimes (a_\T^\ast(\xi_i)-b_\T^\ast(\xi_i))\eta,
\end{align*}
where $(e_i)$ is the canonical orthonormal basis for $\ell^2(I)$.

\emph{Step 1:} We first show that $\norm{S}\leq 2d(\norm{\T})^{1/2}$.

For that purpose let $S_{l}\eta=\sum_{i\in I}e_i\otimes a_\T(\U_{i/2}\xi_i)\eta$ and $S_{r}\eta=\sum_{i\in I}e_i\otimes b_\T(\U_{-i/2}\xi_i)\eta$. With 
\begin{equation*}
    \Pi_i\colon \ell^2(I)\otimes\F_\T(\H)\to \F_\T(\H),\,e_j\otimes\eta\mapsto \delta_{ij}\eta
\end{equation*}
we have $S_{l}^\ast=\sum_{i\in I}a_\T^\ast(\U_{i/2}\xi_i)\Pi_i$ and $S_{r}^\ast=\sum_{i\in I}b_\T^\ast(\U_{-i/2}\xi_i)\Pi_i$.

From \Cref{lem:bound_twisted_norm_tensor_product} we deduce
\begin{align*}
    \left\lVert S_{l}^\ast\left(\sum_{i\in I}e_i\otimes \eta_i\right)\right\rVert_{\F_\T(\H)}^2
    &=\left\lVert\sum_{i\in I} a_\T^\ast(\U_{i/2}\xi_i)\eta_i\right\rVert_{\F_\T(\H)}^2\\
    &=\left\lVert \sum_{i\in I}\U_{i/2}\xi_i\otimes_\phi \eta_i\right\rVert_{\F_\T(\H)}^2\\
    &\leq d(\norm{\T})\left\lVert \sum_{i\in I}\U_{i/2}\xi_i\otimes_\phi \eta_i\right\rVert_{\H\otimes_\phi\F_\T(\H)}^2\\
    &=d(\norm{\T})\sum_{i,k\in I}\langle \eta_i,L_\phi(\U_{i/2}\xi_i)^\ast L_\phi(\U_{i/2}\xi_k)\eta_k\rangle_{\F_\T(\H)}\\
    &=d(\norm{\T})\sum_{i\in I}\norm{p_i\eta_i}_{\F_\T(\H)}^2\\
    &\leq d(\norm{\T})\left\lVert \sum_{i\in I}e_i\otimes \eta_i\right\rVert_{\ell^2(I)\otimes\F_\T(\H)}^2.
\end{align*}
It follows that $\norm{S_{l}}\leq d(\norm{\T})^{1/2}$. An analogous calculation shows $\norm{S_{r}}\leq d(\norm{\T})^{1/2}$. Thus $\norm{S}\leq 2d(\norm{\T})^{1/2}$.

\emph{Step 2:} We prove that $\abs{T}\geq \abs{I}^{1/2}c(\norm{\T})^{1/2}-\abs{I}^{-1/2}d(\norm{\T})^{1/2}$ on $\mathcal F_\T(\H)\ominus L^2(M)$.

For $\eta_1,\dots,\eta_m\in \H$ left-bounded and $\zeta_1,\dots,\zeta_m\in \F_\T(\H)$ let
\begin{align*}    X\sum_{k=1}^m\eta_k\otimes_\phi\zeta_k=\sum_{i\in I}\sum_{k=1}^m L_\phi(\xi_i)^\ast L_\phi(\eta_k)\zeta_k\otimes_\phi\xi_i.
\end{align*}
We have
\begin{align*}
    \left\lVert\sum_{i\in I}\sum_{k=1}^m L_\phi(\xi_i)^\ast L_\phi(\eta_k)\zeta_k\otimes_\phi\xi_i\right\rVert_{\F_\T(\H)\otimes_\phi \H}^2
    &=\sum_{i\in I}\left\lVert\sum_{k=1}^m L_\phi(\xi_i)^\ast L_\phi(\eta_k)\zeta_k\right\rVert_{\F_\T(\H)}^2\\
    &=\sum_{k,l=1}^m\left\langle \zeta_k, L_\phi(\eta_k)^\ast\sum_{i\in I}L_\phi(\xi_i) L_\phi(\xi_i)^\ast L_\phi(\eta_l)\zeta_l\right\rangle_{\F_\T(\H)}
\end{align*}
Since
\begin{equation*}
    \left(\sum_{i\in I}L_\phi(\xi_i) L_\phi(\xi_i)^\ast\right)^2=\sum_{i,j\in I}L_\phi(\xi_i)L_\phi(\xi_i)^\ast L_\phi(\xi_j) L_\phi(\xi_j)^\ast=\sum_{i\in I}L_\phi(\xi_i)L_\phi(\xi_i)^\ast,
\end{equation*}
we have $\sum_{i\in I}L_\phi(\xi_i) L_\phi(\xi_i)^\ast\leq 1$.

Therefore
\begin{align*}
     \left\lVert\sum_{i\in I}\sum_{k=1}^m L_\phi(\xi_i)^\ast L_\phi(\eta_k)\zeta_k\otimes_\phi\xi_i\right\rVert_{\F_\T(\H)\otimes_\phi \H}^2
    &\leq \sum_{k,l=1}^m\left\langle \zeta_k,L_\phi(\eta_k)^\ast L_\phi(\eta_l)\zeta_l\right\rangle_{\F_\T(\H)}\\
    &=\left\langle \sum_{k=1}^m \eta_k\otimes_\phi \zeta_k,\sum_{l=1}^m \eta_l\otimes_\phi\zeta_l\right\rangle_{\H\otimes_\phi\F_\T(\H)}.
\end{align*}
Thus $X$ extends to a contraction from $\H\otimes_\phi\F_\T(\H)$ to $\F_\T(\H)\otimes_\phi \H$. It follows from \Cref{lem:bound_change_twisted_norm} that $\tilde X=j_r X j_l^{-1}$ is a bounded linear operator on $\F_\T(\H)\ominus L^2(M)$ with $\norm{\tilde X}\leq c(\norm{\T})^{-1/2}d(\norm{\T})^{1/2}$.

For $\eta\in \F_\T(\H)$ let
\begin{equation*}
    Y\eta=\sum_{i\in I}e_i\otimes \xi_i\otimes_\phi\eta.
\end{equation*}
We have
\begin{align*}
    \norm{Y \eta}_{\ell^2(I)\otimes\H\otimes_\phi\F_\T(\H)}^2=\sum_{i\in I}\norm{\xi_i\otimes_\phi\eta}_{\H\otimes_\phi\F_\T(\H)}^2=\sum_{i\in I}\norm{\eta}_{\F_\T(\H)}^2=\abs{I}\norm{\eta}_{\F_\T(\H)}^2.
\end{align*}
Thus $Y$ is a bounded operator from $\F_\T(\H)$ to $\ell^2(I)\otimes \H\otimes_\phi\F_\T(\H)$ with norm $\abs{I}^{1/2}$. 
Therefore, another application of \Cref{lem:bound_change_twisted_norm} shows that $\tilde Z=Y^\ast (1\otimes j_l^{-1})$
defines a bounded linear operator from $\ell^2(I)\otimes(\F_\T(\H)\ominus L^2(M))$ to $\F_\T(\H)$ with norm $\norm{\tilde Z}\leq \abs{I}^{1/2}c(\norm{\T})^{-1/2}$.

A direct computation shows that $\tilde ZT=\abs{I}-\tilde X$ on $\F_\T(\H)\ominus L^2(M)$. Hence, whenever $\eta\in \F_\T(\H)\ominus L^2(M)$, then
\begin{align*}
    \abs{I}^{1/2}c(\norm{\T})^{-1/2}\norm{T\eta}\geq \norm{\tilde{Z} T\eta}=\norm{(\abs{I}-\tilde{X})\eta}\geq (\abs{I}-c(\norm{\T})^{-1/2}d(\norm{\T})^{1/2})\norm{\eta}.
\end{align*}
Therefore,
\begin{equation*}
    \abs{T}\geq c(\norm{\T})^{1/2}\abs{I}^{1/2}-d(\norm{\T})^{1/2}\abs{I}^{-1/2}.
\end{equation*}

\emph{Step 3:} We show that $\abs{S+T}\geq c(\norm{\T})^{1/2}\abs{I}^{1/2}-d(\norm{\T})^{1/2}\abs{I}^{-1/2}-2d(\norm{\T})^{1/2}$ on $\F_\T(\H)\ominus L^2(M)$.

If $\eta\in\F_\T(\H)\ominus L^2(M)$, then the inequalities from Step 1 and Step 2 combined imply
\begin{equation*}
    \norm{(S+T)\eta}\geq \norm{T\eta}-\norm{S\eta}\geq (c(\norm{\T})^{1/2}\abs{I}^{1/2}-d(\norm{\T})^{1/2}\abs{I}^{-1/2})\norm{\eta}-2d(\norm{\T})^{1/2}\norm{\eta}.
\end{equation*}

\emph{Step 4:} Let $\kappa=c(\norm{\T})^{1/2}\abs{I}^{1/2}-d(\norm{\T})^{1/2}\abs{I}^{-1/2}-2d(\norm{\T})^{1/2}$. We show that if $x$ is in the centralizer of $\hat\phi_\T$ and $\kappa\geq 0$, then
\begin{equation*}
    \sum_{i\in I}\norm{\F_\T(\J)x^\ast\F_\T(\J)\xi_i -x\xi_i}_{\F_\T(\H)}^2\geq \kappa^2\hat\phi_\T(\abs{x-E_\T(x)}^2).
\end{equation*}
Since $x$ is in the centralizer of $\hat\phi_\T$ and $E_\T$ is a weight-preserving conditional expectation, $E_\T(x)$ is also in the centralizer of $\hat\phi_\T$. Hence we can assume without loss of generality $E_\T(x)=0$. Moreover, the assumption $L_\phi(\xi_i)^\ast L_\phi(\xi_i)=1$ implies
\begin{equation*}
    \phi(1)=\phi(L_\phi(\xi_i)^\ast L_\phi(\xi_i))=\norm{\xi_i}_\H^2<\infty
\end{equation*}
by \Cref{lem:norm_left_bdd_vector}. Let $\Omega\in L^2_+(M)$ denote the cyclic and separating vector representing $\phi$. Since $x$ is in the centralizer of $\hat\phi_\T$, it follows from \Cref{thm:mod_theory_twisted} that $\F_\T(\J)x^\ast \Omega=x\Omega$.

Now we use $\F_\T(\J)x^\ast \F_\T(\J)$ commutes with $a_\T^\ast(\xi_i)+a_\T(\U_{i/2}\xi_i)$ and $x$ commutes with $b_\T^\ast(\xi_i)+b_\T(\U_{-i/2}\xi_i)$ for all $i\in I$ to deduce
\begin{align*}
    \F_\T(\J)x^\ast \F_\T(\J)\xi_i
    &=\F_\T(\J)x^\ast \F_\T(\J)(a_\T^\ast(\xi_i)+a_\T(\U_{i/2}\xi_i)\Omega\\
    &=(a_\T^\ast(\xi_i)+a_\T(\U_{i/2}\xi_i))\F_\T(\J)x^\ast\Omega\\
    &=(a_\T^\ast(\xi_i)+a_\T(\U_{i/2}\xi_i))x\Omega
\end{align*}
and
\begin{equation*}
    x\xi_i=x(b_\T^\ast(\xi_i)+b_\T(\U_{-i/2}\xi_i))(\xi_i)\Omega=(b_\T^\ast(\xi_i)+b_\T(\U_{-i/2}\xi_i))x\Omega.
\end{equation*}
Therefore,
\begin{equation*}
    \sum_{i\in I}\norm{\F_\T(\J)x\F_\T(\J)\xi_i-x\xi_i}_{\F_\T(\H)}^2=\norm{(S+T)x\Omega}_{\F_\T(\H)}^2,
\end{equation*}
and $x\Omega\in L^2(M)^\perp$ because $E_\T(x)=0$. Now the claim follows from Step 3.
\end{proof}

\begin{corollary}\label{cor:infinite-dim_factor}
    Under the assumptions of \Cref{prop:spectral_gap}, the center of the operator-valued twisted Araki--Woods algebra is given by
    \begin{equation*}
        \mathcal Z(\Gamma_\T(\H,\J,(\U_t)))=\{x\in \mathcal Z(M)\mid \xi x=x\xi\text{ for all }\xi\in \H\}.
    \end{equation*}
    In particular, if $M$ is a factor, then $\Gamma_\T(\H,\J,(\U_t))$ is a factor.
\end{corollary}
\begin{proof}
    If $x$ is in the center of $\Gamma_\T(\H,\J,(\U_t))$, then $x$ is in the centralizer of $\hat\phi_\T$. Moreover,
    \begin{align*}
        x\xi&=x (a_\T^\ast(\xi)+a_\T(\J\U_{-i/2}\xi))\Omega\\
        &=(a_\T^\ast(\xi)+a_\T(\J\U_{-i/2}\xi))x\Omega\\
        &=(a_\T^\ast(\xi)+a_\T(\J\U_{-i/2}\xi))\F_\T(\J)x^\ast\F_\T(\J) \Omega\\
        &=\F_\T(\J)x^\ast \F_\T(\J)\xi.
    \end{align*}
    for all $\xi\in \H_0$. Thus $x=E_\T(x)\in M$ by \Cref{prop:spectral_gap}. Hence $x\in \mathcal Z(M)$ and $x\xi=\xi x$ for all $\xi\in \H$ follows from the computation above by the density of $\H_0$ in $\H$.

    The converse inclusion can be shown exactly as in the proof of \Cref{prop:mixing_factor}.
\end{proof}

\section{Application to Quantum Markov Semigroups}\label{sec:applications}

Tomita correspondences occur naturally in the study of generators of GNS-symmetric quantum Markov semigroups and their relation to derivations \cite{Wir22b,Wir24b}. Operator-valued twisted Araki--Woods algebras enter the picture if one wants to obtain derivations with values in a von Neumann superalgebra instead of a bimodule.

In \cite{JRS} in the tracial case and in \cite[Section 7]{Wir22b} in the general case, operator-valued free Araki--Woods algebras were employed for this purpose. One disadvantage is that they are always infinite-dimensional (unless the Tomita correspondence is trivial), even if the base algebra is finite-dimensional. In this section we used a twisted algebra to show that one can always choose the von Neumann superalgebra to be finite-dimensional if the base algebra is a matrix algebra.

Moreover, we use the disintegration theory from \Cref{sec:disintegration_Tomita_corr} to introduce the notion of Bohr spectrum of a quantum Markov semigroup on a semi-finite von Neumann algebra in a way that it coincides with the set of Bohr frequencies occurring in Alicki's theoremm for quantum Markov semigroups on matrix algebras.

Let us briefly recall the relevant terminology regarding quantum Markov semigroups. A \emph{quantum Markov semigroup} on a von Neumann algebra $M$ is a family $(P_t)_{t\geq 0}$ of normal unital completely positive maps on $M$ such that $P_0=\mathrm{id}_M$, $P_sP_t=P_{s+t}$ for all $s,t\geq 0$ and $P_t\to 0$ as $t\to 0$ in the point-weak$^\ast$ topology.

If $\phi$ is a normal semi-finite faithful weight on $M$, then the quantum Markov semigroup $(P_t)$ is called \emph{GNS-symmetric with respect to $\phi$} if $\phi\circ P_t\leq \phi$ and
\begin{equation*}
    \phi(P_t(x)^\ast y)=\phi(x^\ast P_t(y))
\end{equation*}
for all $x,y\in\n_\phi$, $t\geq 0$.

A GNS-symmetric quantum Markov semigroup $(P_t)$ gives rise to a strongly continuous symmetric contraction semigroup $(T_t)$ on $L^2(M)$ characterized by $T_t(\Lambda_\phi(x))=\Lambda_\phi(P_t(x))$ for $x\in \n_\phi$, $t\geq 0$. By semigroup theory, $(T_t)$ is of the form $T_t=e^{-t\L}$ for some positive self-adjoint operator $\L$ on $L^2(M)$. By \cite[Theorem 6.3]{Wir22b}, the set
\begin{equation*}
    \mathfrak a_{\L}=\left\{x\in \bigcap_{z\in \mathbb C}\dom(\sigma^\phi_z): \sigma^\phi_z(x)\in\n_\phi\cap \n_\phi^\ast,\,\Lambda_\phi(\sigma^\phi_z(x))\in \dom(\L^{1/2})\text{ for all }z\in \IC\right\}
\end{equation*}
is a weak$^\ast$ dense $\ast$-subalgebra of $M$ such that $\Lambda_\phi(\mathfrak a_\L)$ is a core for $\mathcal L^{1/2}$.

Following the terminology introduced in the tracial case by Junge, Li and LaRacuente in \cite{LJL20} (see also \cite{BGJ22}), a \emph{derivation triple} for $(P_t)$ is a a triple $(M\subset \hat M,E,\delta)$ consisting of 
\begin{itemize}
    \item a unital inclusion $M\subset \hat M$ of von Neumann algebras,
    \item a faithful normal conditional expectation $E\colon \hat M\to M$, and
    \item a closed operator $\delta\colon \dom(\L^{1/2})\to L^2(\hat M)$ satisfying $\delta\circ J_\phi=J_{\phi\circ E}\circ \delta$, $\delta\circ \Delta_\phi^{it}=\Delta_{\hat\phi}^{it}\circ\delta$ for all $t\in \IR$, and
\begin{equation*}
    \delta(\Lambda_\phi(x)y)=\delta(x\Lambda_\phi^\prime(y))=x\delta(\Lambda_\phi^\prime(y))+\delta(\Lambda_\phi(x))y
\end{equation*}
for $x,y\in \mathfrak a_\L$
\end{itemize}
such that $\L=\delta^\ast \delta$.

It was shown in \cite[Corollary 7.6]{Wir22b} that a GNS-symmetric quantum Markov semigroup admits a derivation triple if and only if it satisfies a property called $\Gamma$-regularity. Every GNS-symmetric quantum Markov semigroup on a finite-dimensional von Neumann algebra or more generally every GNS-symmetric quantum Markov semigroup with bounded generator is $\Gamma$-regular.

\begin{proposition}
If $M$ is a type I factor, $\phi$ a normal semifinite faithful weight on $M$ and $(P_t)$ a quantum Markov semigroup that is GNS-symmetric with respect to $\phi$ and $\Gamma$-regular, then $(P_t)$ admits a derivation triple $(M\subset \hat M,E,\delta)$ with $\hat M=M\overline\otimes \Gamma_T(H,J,(U_t))$ and $E=\mathrm{id}\otimes \psi$, where $\psi$ is the vacuum state on $\Gamma_T(H,J,(U_t))$. Moreover, if $M$ is finite, then $H$ can be chosen finite-dimensional.
\end{proposition}
\begin{proof}
    Let $(M\subset N,F,\partial)$ be any derivation triple for $(P_t)$ and let $\H$ be the closed linear hull of $\{x\partial(a)y\mid x,y\in M,\,a\in\dom(\L^{1/2})\}$. By definition of a derivation triple, $J_{\phi\circ F}$ and $\Delta^{it}_{\phi\circ F}$ leave $\H$ invariant, making it into a Tomita correspondence $(\H,\J,(\U_t))$ over $(M,\phi)$. Moreover, if $M$ is finite-dimensional, so is $\H$.

    By \Cref{thm:Tomita_corr_type_I} we can write $\H=L^2(M)\otimes H$, $\J=J_\phi\otimes J$, $\U_t=\Delta_\phi^{it}\otimes U_t$. Taking $\delta=\partial$ viewed as map into $L^2(M\overline\otimes \Gamma_\T(H,J,(U_t))$ and $\hat M$, $E$ as in the statement of the theorem does the job.
\end{proof}

\begin{corollary}
Every GNS-symmetric quantum Markov semigroup $(P_t)$ on $M_n(\IC)$ admits a derivation triple $(M\subset\hat M,E,\delta)$ with $\hat M$ finite-dimensional.
\end{corollary}
\begin{proof}
If one takes $T$ to be minus the flip map in the previous theorem, then $\Gamma_T(H,J,(U_t))$ is contained in the bounded linear operators on the fermionic Fock space over $H$. In particular, if $H$ is finite-dimensional, so is $\Gamma_T(H,J,(U_t))$ and then $M_n(\IC)\otimes \Gamma_T(H,J,(U_t))$ is also finite dimensional.
\end{proof}

As the previous results show, derivation triples for a given quantum Markov semigroup are not unique, even if one makes the natural minimality assumption that the range of $\delta$ generate $L^2(\hat M)$ as a correspondence. What is unique (up to isomorphism) however is the Tomita correspondence $(\H,\J,(\U_t))$ generated by the range of $\delta$ with the restriction of the modular data from $\hat M$ \cite[Theorem 6.9]{Wir22b}.

In particular, if $M$ is semi-finite with normal faithful trace $\tau$ and $\phi=\tau(h^{1/2}\,\cdot\,h^{1/2})$, then the spectrum of $\U_t(h^{-it}\,\cdot\, h^{it})$ is an invariant of $(P_t)$. This leads us to the following definition.

\begin{definition}
If $(P_t)$ is a GNS-symmetric quantum Markov semigroup and $(\H,\J,(\U_t))$ the associated Tomita correspondence, then the \emph{Bohr spectrum} of $(P_t)$ is the set of all $\omega\in\IR$ for which the operator $\U_t(h^{-it}\,\cdot\,h^{it})-e^{i\omega t}$ on $\H$ is not invertible for some (or all) $t\in\IR$.
\end{definition}

To give an interpretation of the Bohr spectrum, we show how it is related to the Bohr frequencies occurring in Alicki's theorem on generators of GNS-symmetric quantum Markov semigroups on $M_n(\IC)$.

\begin{proposition}
If $h\in M_n(\IC)$ is positive definite and $P_t=e^{-t\L}$ with $\L$ given by
\begin{equation*}
    \L\colon M_n(\IC)\to M_n(\IC),\,\L(x)=\sum_{j=1}^d e^{-\omega_j/2}(v_j^\ast[v_j,x]-[v_j^\ast,x]v_j)
\end{equation*}
with $v_j\in M_n(\IC)$ non-zero and $\omega_j\in \IR$ such that $\{v_j\mid 1\leq j\leq d\}=\{v_j^\ast\mid 1\leq j\leq d\}$ and $hv_j h^{-1}=e^{-\omega_j}v_j$, then $(P_t)$ is a quantum Markov semigroup, which is GNS-symmetric with respect to $\mathrm{tr}(\,\cdot\,h)$, and the Bohr spectrum of $(P_t)$ is $\{\omega_j\mid 1\leq j\leq d\}$.
\end{proposition}
\begin{proof}
For $j\in \{1,\dots,d\}$ let $j^\ast$ denote the unique index in $\{1,\dots,d\}$ such that $v_{j^\ast}=v_j^\ast$. We write $\phi$ for $\mathrm{tr}(h\,\cdot\,)$. As shown in \cite[Section 8.2]{Wir22b}, the Tomita correspondence $(\H,\J,(\U_t))$ associated with $(P_t)$ is given by $\H=\mathrm{HS}(\IC^n)\otimes \IC^d$, $\J=J_\phi\otimes J$ with $(\J \xi)_j=\overline{\xi_{j^\ast}}$, and $\U_t=\Delta_\phi^{it}\otimes U_t$ with  $(U_t \xi)_j=e^{i\omega_j t}\xi_j$. Thus $\V_t=1\otimes U_t$ and the Bohr spectrum of $(P_t)$ is $\{\omega_j\mid 1\leq j\leq d\}$.
\end{proof}

\begin{remark}
By Alicki's theorem (see \cite[Theorem 3]{Ali76}, \cite[Theorem 3.1]{CM17}), every GNS-symmetric quantum Markov semigroup on $M_n(\IC)$ is of the form given in the previous theorem. In this case, the Bohr spectrum is always contained in the spectrum of $\log \Delta_\phi$. It is an interesting question whether this is true for the Bohr spectrum of GNS-symmetric quantum Markov semigroups on an arbitrary semi-finite von Neumann algebras.
\end{remark}


\printbibliography
\end{document}